\documentclass{siamltex}
\usepackage{lineno,hyperref}
\usepackage{multirow}
\usepackage[english]{babel}
\usepackage{amsmath}
\usepackage{bm}
\usepackage{mathrsfs}
\usepackage{calligra}
\usepackage{calrsfs}
\usepackage{amsfonts}
\usepackage{amssymb}
\usepackage{graphicx}
\usepackage{tikz}
\usepackage{here}
\usepackage{enumitem}
\usepackage{algorithm}
\usepackage{algorithmic}
\usepackage{graphicx}
\usepackage{amsmath}
\usepackage{here}
\usetikzlibrary{matrix,calc}
\usepackage{caption}

\newtheorem{remark}{Remark}

\DeclareMathOperator*{\argmin}{\arg\,\min}

\title{ Trace ratio based manifold learning with tensor data }

\author{M. Bouallala \footnotemark[3] \thanks{ LMPA, 50 rue 
		F. Buisson, Universit\'e du Littoral Cote d'Opale, 62228 Calais-Cedex, France.}
	\and
	F. Dufrenois\thanks{ LISIC, 50 rue 
		F. Buisson, Universit\'e du Littoral Cote d'Opale, 62228 Calais-Cedex, France.}
	\and 
	K. Jbilou\footnotemark[1] \thanks{The Vanguard Center, Mohammed VI Polytechnic University, Green 
		City, Benguérir Morocco.}
	\and
	A. Ratnani\footnotemark[3]
}
\begin{document}
	\maketitle 
	
	\begin{abstract}
	In this paper, we propose an extension of trace ratio based Manifold learning methods to deal with multidimensional data sets. Based on recent progress on the  tensor-tensor product, we present a generalization of the trace ratio criterion by using the properties of the t-product. This will conduct us to introduce some new concepts such as Laplacian tensor and we will study formally the trace ratio problem by discuting the conditions for the exitence of solutions and optimality. Next, we will present a tensor Newton QR decomposition algorithm for solving the trace ratio problem. Manifold learning  methods such as Laplacian eigenmaps, linear discriminant analysis and locally linear embedding will be formulated in a tensor representation and optimized by the proposed algorithm. Lastly, we will evaluate the performance of the different studied dimension reduction methods on several synthetic and real world data sets.
	\end{abstract}
	
		\section{Introduction} 
	 In the big data era, machine learning methods (ML) are faced with an increasing volume of data which can both mix different modalities and contain several thousands, or millions of features. This unstoppable curse of dimensionality is the « Achilles heel » of most ML methods which involves an increase of the complexity of the model and a loss of generalization capacity. Dimensionality reduction (DR) or more generally manifold learning is a tailored response to minimize this problem and open up the access to modern real world applications such as multiview classification \cite{zhang2016flexible}, object detection \cite{kokiopoulou2011trace, zhou2014dimension},… The principle of DR consists in projecting high-dimensional data into a lower-dimensional space dimensional while retaining as much of the important information as possible. DR includes a wide variety of methods, from the most classical and popular such as principal component Analysis (PCA,\cite{anowar2021conceptual, kokiopoulou2011trace}), Linear discriminant analysis (LDA,\cite{anowar2021conceptual, kokiopoulou2011trace}), singular value decomposition (SVD,\cite{anowar2021conceptual}),... to the most recent Self Organizing Map (SOM,\cite{kohonen2013essentials}), ISOMAP \cite{anowar2021conceptual, kokiopoulou2011trace}, Locally Linear Embedding (LLE,\cite{kokiopoulou2011trace,ghojogh2020locally,saul2000introduction}), Laplacian Eigenmaps (LE,\cite{carreira2007laplacian, kokiopoulou2011trace}),… to name but a few (for a review see \cite{anowar2021conceptual,kokiopoulou2011trace}).\\ 
		A wide majority of these dimensionality reduction methods are formulated under the form of a ratio trace problem whose the optimization amounts to solve a generalized eigenproblem. Although all these methods have been developed in a matricial form (second-order tensor), they are unsuited and loss their efficiency for large multidimensional data sets. Representing data and formulating an optimization problem with tensors of order greater than 2 become a new challenging task for modern ML methods. \\
		Recently, Principal Component Analysis (PCA) and Linear Discriminant Analysis (LDA) have been generalized to deal with multidimensional data sets. Firstly, by using the n-mode product of tensors, numerous optimizations procedures have been proposed to solve PCA and LDA in this context \cite{lu2008uncorrelated,lu2008mpca,yan2006multilinear}. However, these approaches are not « fully » in a tensor form since the underlying optimization process amounts to find projectives matrices instead of projectives tensors.  Recently, based on recent developments on tensor-tensor products \cite{kilmer2011factorization}, this question have been solved. For instance, Sparse Regularization Tensor Robust Principal Component Analysis \cite{yang2020sparse}, and Multilinear Discriminant Analysis (MLDA,\cite{dufrenois2022multilinear}) propose to solve the DR problem by using the properties of the t-product. The particularity of this t-product is to realize the optimization of the trace ratio problem associated to PCA and LDA in a so called « transform domain » where the tensor-tensor product can be defined. The transformation being invertible, a projective tensor, solution of the problem, is next recovered. \\
		Based on these recent developments, we propose here to generalize several Manifold learning methods formulated as a trace ratio criterion. In particular, our study will focus on the generalization of the following approaches: Local Discriminant Embedding (LDE) \cite{chen2005local,zhou2014dimension}, Laplacian Eigenmaps (LE) \cite{belkin2001laplacian,carreira2007laplacian, kokiopoulou2011trace}, and Locally Linear Embedding (LLE) \cite{kokiopoulou2011trace,ghojogh2020locally, saul2000introduction}. These methods are central statisticals tools in the ML toolbox and have demonstrated they superiority over conventional methods such as PCA and LDA \cite{kilmer2011factorization}.	The particularity of these methods is to be based on three key steps:  a) building a neighborhood graph,b) computing a weighting vector and c) computing the embedding. Below, let us recall briefly the principle of these methods.\\
		- LE is an unsupervised manifold learning method which relies on the construction of a graph from neighborhood information of the data set and the minimization of given cost function based on this graph. This method ensures that points close to each other on the manifold are mapped close to each other in the low-dimensional space, thus preserving local distances. \\
		- LLE is also an unsupervised dimensionality reduction method which tries as LE to preserve the local structure of data in the embedding space. The principle of LLE is fitting the local structure of manifold in the embedding space. The local structure of the data is obtained by building a k-NN graph. \\
		- LDE is a supervised manifold learning algorithm which makes use of both the neighborhood relationships between data points and the class label information to obtain a lower-dimensional embedding. Unlike LDA and related methods, the discrimination ability of LDE does not strongly depend on the data distribution, such as the Gaussian assumption. Moreover, unlike many manifold learning methods such as Isomap \cite{anowar2021conceptual} and locally linear embedding (LLE) \cite{ghojogh2020locally}, LDE uses label information to find the embedding and can naturally handle new test data in classification applications.  \\
		\\
		Our study will based on several contributions: 
		First, the generalization of theses methods to tensor data will conduct us to define some key concepts and properties associated to the t-product such as  trace of tensor, Laplacian tensor, positive definite and semi-definite tensor,...Moreover, we will present a theoretical examination of the trace-ratio tensor problem, discussing both the existence of solutions and optimality conditions.  We will develop the Tensor Newton-QR algorithm as a new approach for solving the trace ratio tensor problem. Lastly, we will formulate LE, LLE and LDE as a new trace ratio criterion based on tensor representation.\\
		\\
	The organization of the paper is as follows. In Section \ref{s2}, we present an overview of multilinear algebra concepts. In Section \ref{s3}, we study optimization problems related to trace-ratio tensor methods. Section \ref{s4} introduces the trace-ratio tensor methods using t-product. In Section \ref{s5}, we compare our approach with the state of the art. Section \ref{s6} concludes the paper.

	\section{Multilinear algebra concepts}\label{s2} {A tensor is a mathematical object that represents a wide range of data, including scalar values, vectors, matrices, and higher-dimensional arrays. A first-order tensor can be seen as a vector and a second-order tensor as a matrix. Let $\mathcal{A}\in \mathbb{R}^{n_{1} \times n_{2} \times n_{3}}$ be a third-order tensor, its $(i,j,k)$-th element is denoted by $\mathcal{A}_{ijk}$}. We can extract a fiber of $\mathcal{A}$ by fixing two indices, say $j$ and $k$. The column, line, and tube fibers are denoted by $\boldsymbol{a}_{.jk}$, $\boldsymbol{a}_{i.k}$, and $\boldsymbol{a}_{ij.}$, respectively. Similarly, we can define a slice of $\mathcal{A}$ by fixing one index out of three. {For a third-order tensor there are three modes of slices: horizontal (mode 1), lateral (mode 2), and frontal (mode 3) slides, represented by $\mathcal{A}_{i::}$, $\mathcal{A}_{:j:}$, and $\mathcal{A}_{::k}$, respectively.}
Further, $\mathbb{R}^{n_1 \times n_2\times n_3}$ denotes the space of {reel} third-order tensors of size $n_1 \times n_2 \times n_3,\ \mathbb{R}^{n_1 \times 1 \times n_3}$ stands for the space of lateral slices of size $n_1 \times 1 \times n_3$, and $\mathbb{R}^{1 \times 1 \times n_3}$ denotes the space of tubes with $n_3$ entries. {For more details see the work done by Kolda et al. \cite{kolda2009tensor}}.

Consider a tensor $\mathcal{A} \in \mathbb{R}^{n_1 \times n_2 \times n_3}$, the Frobenius norm of the tensor $\mathcal{A}$ can be expressed as follows
\[
\|\mathcal{A}\|_F=\sqrt{\sum_{i=1}^{n_1} \sum_{j=1}^{n_2} \sum_{k=1}^{n_3} \mathcal{A}_{i j k}^2},
\] 
and its associated inner product between two third-order tensors $\mathcal{A}$ and $\mathcal{B}$ in $\mathbb{R}^{n_1 \times n_2 \times n_3}$ is defined by
\[
\langle\mathcal{A}, \mathcal{B}\rangle=\sum_{i, j, k=1}^{n_1, n_2, n_3} \mathcal{A}_{i, j, k} \mathcal{B}_{i, j, k}.
\]
We recall the Kronecker product, given in the following definition.
\begin{definition}[Kronecker product] \label{def2}
The Kronecker product of matrices ${A} \in \mathbb{R}^{I \times J}$ and ${B} \in \mathbb{R}^{K \times L}$ is denoted by ${A} \otimes {B}$. The result is a matrix of size $(I K) \times(J L)$ and defined by
	$$
	\begin{aligned}
		{A} \otimes {B} & =\begin{pmatrix}
			a_{11} {B} & a_{12} {B} & \cdots & a_{1 J} {B} \\
			a_{21} {B} & a_{22} {B} & \cdots & a_{2 J} {B} \\
			\vdots & \vdots & \ddots & \vdots \\
			a_{I 1} {B} & a_{I 2} {B} & \cdots & a_{I J} {B}
		\end{pmatrix} \\
		& =\begin{pmatrix}
			{a}_1 \otimes {b}_1 & {a}_1 \otimes {b}_2 & {a}_1 \otimes {b} & \cdots & {a}_J \otimes {b}_{L-1} & {a}_J \otimes {b}_L
		\end{pmatrix}.
	\end{aligned}
	$$
\end{definition}
We define several important concepts. We start with the matricization of tensors, also known as unfolding or flattening. It consists of reordering the elements of a tensor into a matrix, for more details see \cite{kolda2009tensor}.
\begin{definition}\label{def1}
	Let $\mathcal{A} \in \mathbb{R}^{n_{1} \times n_{2} \times \ldots \times n_{N}}$, flattening $\mathcal{A}$ along the $k^{th}$ mode or the $k$-mode matricization of $\mathcal{A}$ gives a matrix denoted by $\mathcal{A}_{(k)}$ which consists in arranging the $k$-mode fibers to be the columns of the resulting matrix.
	Tensor element $\left(i_{1}, i_{2}, \ldots, i_{N}\right)$ maps to matrix element $\left(i_{k}, j\right)$, where
	\[
	j=1+\sum_{\substack{k=1 \\ k \neq n}}^{N}\left(i_{k}-1\right) J_{k}, \quad \text { with } \quad J_{k}=\prod_{\substack{m=1 \\ m \neq n}}^{k-1} n_{m}.
	\]
\end{definition}
The $k$-mode product \cite{kolda2009tensor} is defined as follows 
\begin{definition}\label{def4}
The ${k}$-mode product of $\mathcal{A} \in \mathbb{R}^{n_{1} \times n_{2} \times \ldots \times n_{N}}$ with a matrix $U \in \mathbb{R}^{m \times n_{k}}$ is a new tensor $\mathcal{B} \in \mathbb{R}^{n_{1} \times \ldots \times n_{k-1} \times {m} \times n_{k+1} \times \ldots \times n_{N}}$ defined by
	\[
	\left(\mathcal{A} \times_k {U}\right)_{i_1 \cdots i_{k-1} j i_{k+1} \cdots i_N}=\sum_{i_k=1}^{n_k} \mathcal{A}_{i_1 i_2 \cdots i_N} U_{j i_k}.
	\]
 \end{definition}
Let $\mathcal{A}$ and $U$ be a tensor and a matrix of appropriate sizes. The relation between the matricization and the $k$-mode product is given by the following equivalence.
	\[
	\mathcal{B}=\mathcal{A} \times_{k} U \Longleftrightarrow 
	\mathcal{B}_{(k)}=U \mathcal{A}_{(k)},
 \]
	where $\mathcal{A}_{(k)}$ and $\mathcal{B}_{(k)}$ denotes the $k$-mode matricization of $\mathcal{A}$ and $\mathcal{B}$, respectively.
 
The face-wise product has been used by Kilmer et al in \cite{kernfeld2015tensor}.
\begin{definition} [ face-wise product] \label{def3}
	Let $\mathcal{A} \in \mathbb{R}^{n_1\times l\times n_3}$ and $\mathcal{B} \in \mathbb{R}^{l\times n_2 \times n_3}$ be two third-order tensors. Then the face-wise product of $\mathcal{A}$ and $\mathcal{B}$ is the tensor of size $n_1 \times n_2 \times n_3$ whose $i$-th frontal slice is given from the product of the $i$-th frontal slices of $\mathcal{A}$ and $\mathcal{B}$, i.e.,
	$$
	(\mathcal{A} \triangle \mathcal{B})^{(i)}=\mathcal{A}^{(i)} \mathcal{B}^{(i)}.
	$$
\end{definition}
\subsection{The t-product} The t-product is a tensor-tensor product that has been introduced by Kilmer and her collaborators in \cite{kilmer2011factorization}. This product was only restricted to third-order tensors.\\
To introduce the t-product we need firstly to define some specific block matrices.
\begin{itemize}
    \item The block circulant matrix associated with $\mathcal{A} \in \mathbb{R}^{n_1 \times n_2 \times n_3}$ 
\begin{equation}\label{equ2.1}
	\text{\tt bcirc}(\mathcal{A})=\begin{pmatrix}
		\mathcal{A}^{(1)} & \mathcal{A}^{(n_3)} & \ldots & \mathcal{A}^{(2)} \\
		\mathcal{A}^{(2)} & \mathcal{A}^{(1)} & \ldots & \mathcal{A}^{(3)} \\
		\vdots & \ddots & \ddots & \vdots \\
		\mathcal{A}^{(n_3)} & \mathcal{A}^{(n_3-1)} & \ldots & \mathcal{A}^{(1)}
	\end{pmatrix} \in \mathbb{R}^{n_1 n_3 \times n_2 n_3} .
\end{equation}
\item The operator unfold applied to $\mathcal{A}$ gives the matrix made up of its frontal slices,
\[
\text{\tt unfold}(\mathcal{A})=\begin{pmatrix}
	\mathcal{A}^{(1)} \\
	\mathcal{A}^{(2)} \\
	\vdots \\
	\mathcal{A}^{(n_3)}
\end{pmatrix} \in \mathbb{R}^{n_1 n_3 \times n_2} .\]
We also will need the inverse operator \text{\tt fold} such that \text{\tt fold}(\text{\tt unfold} $(\mathcal{A}))=\mathcal{A}$.
\item The block diagonal matrix associated with $\mathcal{A}$ is defined as
\[
{\tt bdiag}(\mathcal{A})=\begin{pmatrix}
	\mathcal{A}^{(1)} & & & \\
	& \mathcal{A}^{(2)} & & \\
	& & \ddots & \\
	& & & \mathcal{A}^{(n_3)}
\end{pmatrix} \in \mathbb{R}^{n_1 n_3 \times n_2 n_3} .
\]
\end{itemize}
The t-product is given in the following definition.
\begin{definition}\label{def5}
	 Let $\mathcal{A} \in \mathbb{R}^{n_{1} \times q \times n_3 }$ and $\mathcal{B} \in \mathbb{R}^{q \times n_2 \times n_3 }$ be two third-order tensors. The t-product between $\mathcal{A}$ and $\mathcal{B}$ is defined by
	\[
	\mathcal{A} \star \mathcal{B}:={\tt fold}\left({\tt bcirc}\left(\mathcal{A}\right){\tt unfold}\left(\mathcal{B}\right)\right) \in \mathbb{R}^{n_{1} \times n_2 \times n_3}.
 \]
\end{definition}
According to \cite{kilmer2011factorization}, the discrete Fourier transform (DFT) can block-diagonalize the block circulant matrix \eqref{equ2.1}, i.e.,
\[
{\tt bcirc}\left(\mathcal{A}\right)=\left(F_{n_3}^H \otimes I_{n_1}\right) \text{\tt bdiag}(\widehat{\mathcal{A}})\left(F_{n_3} \otimes I_{n_2}\right),
\]
where $F_{n_3} \in \mathbb{C}^{n_3 \times n_3}$ is the discrete Fourier matrix, $F_n^{H}$ denotes its hermitian transpose, $\widehat{\mathcal{A}}$ stands for the Fourier transform of $\mathcal{A}$ along each tube, $I_{n_1} \in \mathbb{R}^{n_1 \times n_1}$ denotes the identity matrix. The tensor $\widehat{\mathcal{A}}$, can be computed with the fast Fourier transform (FFT) algorithm; see \cite{kilmer2011factorization} for details.
Using MATLAB notations, we have
\[
\widehat{\mathcal{A}}={\tt fft}\left(\mathcal{A},[\,], 3\right),
\]
The command for the inverse operation is
\[
\mathcal{A}={\tt ifft}\left(\widehat{\mathcal{A}},[\,], 3\right).
\]
Hence, according to \cite{kilmer2011factorization}, the t-product $\mathcal{C}=\mathcal{A} \star \mathcal{B}$ can be evaluated as
\begin{equation}\label{equ22}
\widehat{\mathcal{C}}^{(i)}=\widehat{\mathcal{A}}^{(i)} \widehat{\mathcal{B}}^{(i)}, \quad i=1,2, \ldots, n_3,
\end{equation}
where $\widehat{\mathcal{A}}^{(i)}, \widehat{\mathcal{B}}^{(i)}$, and $\widehat{\mathcal{C}}^{(i)}$ are the $i$-th frontal slices of the tensors $\widehat{\mathcal{A}}, \widehat{\mathcal{B}}$, and $\widehat{\mathcal{C}}$, respectively.\\
As mentioned earlier by Kilmer et al \cite{kilmer2013third}, the Discrete Fourier Transform (DFT) is symmetric when used with real data, which makes it easier to calculate the t-product using the FFT. This is explained in more detail in the following lemma.
\begin{lemma}[\cite{rojo2004some}]
	Given a real vector $v \in \mathbb{R}^{n}$, the corresponding DFT vector $\widehat{v}=F_{n} v$ satisfies
	\[
 \widehat{v}_1 \in \mathbb{R}, \quad {\tt conj}\left(\widehat{v}_i\right)=\widehat{v}_{n-i+2}, \quad i=2,3, \ldots,\left[\frac{n+1}{2}\right],
	\]
	In this context, {\tt conj} is used to represent the complex conjugation operator, while  $\left[\frac{n_{3}+1}{2}\right]$ indicates the integer part of $\frac{n_{3}+1}{2}$.
\end{lemma}

\noindent It follows that for a third-order tensor $\mathcal{A} \in \mathbb{R}^{n_{1} \times n_2\times n_3}$, we have
\[
\widehat{\mathcal{A}}^{(1)} \in \mathbb{R}^{n_{1} \times n_2}, \quad \text{\tt conj}\left(\widehat{\mathcal{A}}^{(i)}\right)=\widehat{\mathcal{A}}^{(n_3-i+2)}, \quad i=2,3, \ldots,\left[\frac{n_3+1}{2}\right].
\]
This shows that the t-product of two third-order tensors can be determined by evaluating just about half the number of products involved in \eqref{equ22}. Algorithm \ref{algo1} describes the computations.
\begin{algorithm}[H]
	\caption{t-product of third-order tensors}
	\label{algo1}
	\begin{algorithmic}
		\STATE{\textbf{Input:} $\mathcal{A} \in \mathbb{R}^{n_1 \times q \times n_3 }, \mathcal{B} \in \mathbb{R}^{q \times n_2 \times n_3 }$.}
		\STATE{\textbf{Output:} $\mathcal{C}:=\mathcal{A} \star \mathcal{B} \in \mathbb{R}^{n_1 \times n_2 \times n_3}$.}\\
		\STATE{Compute $\widehat{\mathcal{A}}={\tt fft}(\mathcal{A},[\,], 3), \widehat{\mathcal{B}}={\tt fft}(\mathcal{B},[\,], 3)$.}
		\FOR{$i=1, \ldots,\left[\frac{n_3+1}{2}\right]$}
  		\STATE{$ \widehat{\mathcal{C}}^{(i)}=\widehat{\mathcal{A}}^{(i)} \widehat{\mathcal{B}}^{(i)}$.}
		\ENDFOR
		\FOR{$i=\left[\frac{n_3+1}{2}\right]+1, \ldots, n_3$}
		\STATE{$\widehat{\mathcal{C}}^{(i)}=\text{\tt conj}\left(\widehat{\mathcal{C}}^{(n_3-i+2)}\right)$}.
		\ENDFOR
		\STATE{$\mathcal{C}={\tt ifft}(\widehat{\mathcal{C}},[\,], 3)$}.
	\end{algorithmic}
\end{algorithm}
The following definition is concerned with the t-product of a third-order tensor and a tube.
\begin{definition}\label{def6} \cite{el2023svds}
Given a tensor $\mathcal{A} \in \mathbb{R}^{n_1 \times n_2 \times n_3}$ and a tube $\boldsymbol{b} \in \mathbb{R}^{1\times 1 \times n_3}$, we define $\mathcal{C} := \mathcal{A} \star \boldsymbol{b} \in \mathbb{R}^{n_1 \times n_2 \times n_3}$. This results from applying the inverse DFT (Discrete Fourier Transform) along each tube of $\widehat{\mathcal{C}}$.
 where each frontal slice is determined by the usual matrix product between each frontal slice of $\widehat{\mathcal{A}}$ and $\widehat{\boldsymbol{b}}$, i.e.,
 \[
\widehat{\mathcal{C}}^{(i)}=\widehat{\mathcal{A}}^{(i)}\widehat{\boldsymbol{b}}^{(i)}=\widehat{\boldsymbol{b}}^{\left(i\right)} \widehat{\mathcal{A}}^{(i)}, \quad i=1,2, \ldots, n_3.
	\]
\end{definition}
From \cite{lu2019tensor}, we have the following relation
\begin{equation}\label{fft}
 \mathcal{A}^{(1)} =\dfrac{1}{n_3} \sum_{i = 1}^{n_3} \widehat{\mathcal{A}}^{(i)}.   
\end{equation}

Next, we give some definitions of some specific tensors. All these notions are in \cite{kilmer2011factorization}
\begin{definition}[Identity tensor]\label{def7}
	The Identity tensor $\mathcal{I}_{n} \in \mathbb{R}^{n \times n\times n_3}$ is the tensor whose first frontal slice is the $n \times n$ identity matrix, and whose other frontal slices are all zeros.
\end{definition}
\begin{definition}[Tensor transpose]\label{def8}
	The transpose of a real third-order tensor, $\mathcal{A} \in \mathbb{R}^{n_1 \times n_2\times n_3}$, denoted by $\mathcal{A}^T \in \mathbb{R}^{n_2 \times n_1\times n_3}$, is the tensor obtained by first transposing each one of the frontal slices of $\mathcal{A}$, and then reversing the order of the transposed frontal slices 2 through $n_3$. {Moreover, for a square tensor $\mathcal{A}\in \mathbb{R}^{n\times n\times n_3}$, $\mathcal{A}$ is considered f-symmetric if $\mathcal{A} = \mathcal{A}^{\top}$.}\\
	Let the third-order tensors $\mathcal{A}$ and $\mathcal{B}$ be such that the products $\mathcal{A} \star \mathcal{B}$ and $\mathcal{B}^{T} \star \mathcal{A}^{T}$ are defined. Then, similarly to the matrix transpose, the tensor transpose satisfies $(\mathcal{A} \star \mathcal{B})^{T}=$ $\mathcal{B}^{T} \star \mathcal{A}^{T}$.\\
 A tensor $\mathcal{Q} \in \mathbb{R}^{n \times n\times n_3}$ is said to be f-orthogonal (or f-unitary) if and only if
	\[
	\mathcal{Q}^{T} \star \mathcal{Q}=\mathcal{Q} \star \mathcal{Q}^{T}=\mathcal{I}_{n} .
	\]
\end{definition}
\begin{definition}[Positive Definite (Semi-Definite) Tensor]\label{def9}
	Consider the tensor $\mathcal{A} \in\mathbb{R}^{n \times n \times n_3}$, the tensor $\mathcal{A}$ is positive definite (semi-definite) if and only if each frontal slice $\widehat{\mathcal{A}}^{(i)}$ is positive definite (semi-definite). 
\end{definition}

\begin{definition}[Laplacian tensor]
Consider the tensor \(\mathcal{A} \in \mathbb{R}^{n \times n \times n_3}\). The tensor \(\mathcal{A}\) is called a Laplacian tensor if each of frontal slices \(\widehat{\mathcal{A}}^{(i)}\) is a Laplacian matrix.
\end{definition}\\
Note that the Laplacian tensor is f-symmetric semi-definite.

{
\begin{remark}\label{remark}
    A f-symmetric positive definite tensor is invertible.
\end{remark}
The proof can be seen easily from the fact that each frontal slice in the Fourier space of a tensor is positive definite.
}
\begin{definition}[Inverse of tensor]\label{def10}
	We say that the tensor $\mathcal{A} \in \mathbb{R}^{n \times n\times n_3}$ is nonsingular {(invertible)} if there exists a tensor $\mathcal{B} \in \mathbb{R}^{n \times n \times n_3 }$ such that the following conditions hold
	\[
	\mathcal{B} \star \mathcal{A}=\mathcal{A} \star \mathcal{B}=\mathcal{I}_{n},
	\]
	where $\mathcal{B}$ is the inverse of the tensor $\mathcal{A}$, denoted by $\mathcal{A}^{-1}$, and $\mathcal{I}_{n}$ is the Identity tensor of size $n \times n \times n_3$.
\end{definition}
\begin{definition}[Tensor trace]\label{def11}
	Consider a tensor $\mathcal{A} \in \mathbb{R}^{n \times n\times n_3}$, its trace can be defined by
 {
	\[
	\operatorname{Trace}({\mathcal{A}}) = \frac{1}{n_3}\sum_{i=1}^{n_3} \operatorname{Trace} \left(\widehat{\mathcal{A}}^{(i)}\right).
	\]
 }
\end{definition}
According to Equation \eqref{fft}, the trace of a third-order tensor satisfies also the following relation 
{
\[
Trace\left(\mathcal{A}\right)= \operatorname{Trace}\left(\mathcal{A}^{(1)}\right).
\]
}
\begin{proposition}\label{pro1}
	Consider a tensor $\mathcal{A} \in \mathbb{R}^{n_1 \times n_2 \times n_3}$,  the Frobenius norm of the tensor $\mathcal{A}$ can be expressed as follows
 {
 \[
\|\mathcal{A}\|_F^{2}=Trace\left(\mathcal{A} \star \mathcal{A}^{T}\right).
\]
}
\end{proposition}
\begin{proof}
    Form \cite{lu2019tensor} and the definition of Frobenius norm we have
    {
\[
\begin{aligned}
\|\mathcal{A}\|_F^2 & =\frac{1}{n_3} \|\widehat{\mathcal{A}}\|_F^2 \\
& = \frac{1}{n_3} \sum_{i=1}^{n_3}\|\widehat{\mathcal{A}}^{(i)}\|_F^2 \\
& =\frac{1}{n_3}  \sum_{i=1}^{n_3} Trace\left(\widehat{\mathcal{A}}^{(i) }\widehat{\mathcal{A}}^{(i)\top}\right)\\
& = \frac{1}{n_3}  \sum_{i=1}^{n_3} Trace\left(\widehat{\left(\mathcal{A} \star \mathcal{A}^{\top}\right)^{(i)}}\right) =  Trace({\mathcal{A} \star \mathcal{A}^{\top}}).
\end{aligned}
\]
}
\end{proof}
{It can be also seen that 
\[
\left<\mathcal{A},\mathcal{B}\right>=Trace\left(\mathcal{A}\star \mathcal{B}^T\right)=Trace\left(\mathcal{A}^T\star \mathcal{B}\right).
\]}
\begin{definition}[F-diagonal tensor]\label{def13}
	{Consider the tensor $\mathcal{A} \in \mathbb{R}^{n \times n \times n_3}$, the tensor $\mathcal{A}$ is f-diagonal if each frontal slice $\mathcal{A}^{(i)}$ is diagonal, and the diagonal of the tensor $\mathcal{A}$ can be represented by a matrix of size $n \times n_3$, where each column contains the diagonal elements of the corresponding frontal slice $\mathcal{A}^{(i)}$, and it is denoted by $\text{Diag$\left(\mathcal{A}\right)$}$.}
\end{definition}

The rank of a tensor is defined dependently on the type of the tensor product used. There have been defined ranks linked with the t-product, and each one is associated with a specific application, see \cite{lu2019tensor,kernfeld2015tensor}.
\begin{definition}[The tubal rank]
Let $\mathscr{A} \in \mathbb{R}^{\ell \times p \times n_3}$ be a third-order tensor, its tubal rank is defined as
\[
\operatorname{rank}_t(\mathscr{A})=\operatorname{card}\left\{\sigma_i \neq 0, \quad i=1,2, \ldots, \min \{\ell, p\}\right\},
\]
where $\sigma_i$ is the norm of the singular tube $s_i$ given from t-svd of $\mathcal{A}$, for more details see \cite{kernfeld2015tensor,lu2019tensor,el2023svds}.
\end{definition}

The range and the null space of a third-order tensor under the t-product have been defined by El hachimi et al \cite{el2023spectral}. Let $\mathcal{A}\in \mathbb{R}^{n_1\times n_2\times n_3}$ be a third-order tensor, then
\[
\text{Range}\left(\mathcal{A}\right)=\{\mathcal{A}\star \mathcal{X}/\, \mathcal{X}\in \mathbb{R}^{n_2\times 1\times n_3}\},
\]
and the null space of $\mathscr{A}$ is defined by
\[
\begin{gathered}
\operatorname{Null}_{\mathrm{t}}(\mathscr{A})=\left\{\mathscr{X}_1, \ldots, {\mathscr{X}}_r \in \mathbb{R}^{p\times 1 \times n_3}: \widehat{\mathscr{X}}_j^{(i)} \in \operatorname{Null}\left(\widehat{\mathscr{A}}^{(i)}\right) \text { for } i=1,2, \ldots, n_3,\right. \\
\text { with } \left.\left\|\widehat{\mathscr{X}}_j^{(i)}\right\|_F \geqslant\left\|\widehat{\mathscr{X}}_{j+1}^{(i)}\right\|_F, j=1,2, \ldots,{ r_i}\right\},
\end{gathered}
\]
where $r_i=dim\left(Null\left(\widehat{\mathcal{A}}^{(i)}\right)\right)$ and  $r=\min_{1\leq i\leq n_3}r_i$. For further details, see \cite{el2023spectral}.

The generalization of eigenvalues and eigenvectors was presented and detailed in \cite{el2023spectral}. This concept will be essential in the coming section. 
\begin{definition}
Consider the tensor $\mathcal{B} \in \mathbb{R}^{n \times n \times n_3}$. {A tube $\theta \in \mathbb{C}^{1 \times 1 \times n_3}$} is recognized as an eigentube of $\mathcal{B}$ corresponding to a specific lateral slice {$ \vec{\mathcal{V}} \not \equiv \mathbf{0}\in\mathbb{C}^{n \times 1 \times n_3}$}, if they satisfy  the condition
\[
                \mathcal{B} \star \ \vec{\mathcal{V}} =  \vec{\mathcal{V} }\star \theta,
\]
the lateral slice $ \vec{\mathcal{V}}$ is identified as an eigenslice or right eigenslice of $\mathcal{B}$ associated with $\theta$. Furthermore, the pair $\{\theta,  \vec{\mathcal{V}}\}$ is referred to as an eigenpair of $\mathcal{B}$.
With $\mathcal{B} \not \equiv \boldsymbol{0}$ implies that $\widehat{\mathcal{B}}^{(i)} \neq 0$ for every $i=1,2, \ldots, n_3$.
\end{definition}

Authors in \cite{el2023spectral} have defined a specific order for eigentubes. Consider a third-order tensor $\mathcal{B} \in \mathbb{R}^{n \times n \times n_3}$. For each $k = 1, 2, \ldots, n_3$, the eigenvalues of $\widehat{\mathcal{B}}^{(k)}$ are denoted as $\delta_{1, k}, \delta_{2, k}, \ldots, \delta_{n, k}$. These eigenvalues are arranged such that
\[
\left|\delta_{l, k}\right| \geq \left|\delta_{l+1, k}\right|, \quad l=1,2, \ldots, n-1.
\]
The sequence of ordered eigentubes $\mu_1, \mu_2, \ldots, \mu_n$ of $\mathcal{B}$ is then defined as
\[
\widehat{\mu}_l^{(k)}=\delta_{l, k}, \quad k=1,2, \ldots, n_3, \quad l=1,2, \ldots, n,
\]
{Note that a tensor $\mathcal{B}^{n \times n \times n_3}$ admits at most $n$ ordered eigentubes. In our work, we are interested in this eigentubes}; see \cite{el2023spectral} for more details.
\begin{definition}[Tensor f-diagonalization]
\label{def16}
	{Let $\mathcal{A} \in \mathbb{R}^{n \times n \times n_3}$. $\mathcal{A}$ is said to be f-diagonalizable if it is similar to an  f-diagonal tensor, i.e., 
 \[
\mathcal{A}=\mathcal{V}\star \mathcal{D}\star \mathcal{V}^{-1},
 \]
 for some invertible tensor $\mathcal{V}\in \mathbb{R}^{n\times n\times n_3}$ and an f-diagonal tensor $\mathcal{D}\in \mathbb{R}^{n\times n\times n_3}$. In this case, $\mathcal{V}$ and $\mathcal{D}$ contain the eigenslices and the eigentubes, respectively, of $\mathcal{A}.$}
\end{definition}

\section{Trace-Ratio Tensor problem}\label{s3}
The trace ratio tensor problem is an important concept in machine learning for tasks such as feature extraction and dimensionality reduction, as it aids in the analysis of complicated, multidimensional data and can be expressed as follows
\begin{equation}\label{equ34}
	\max _{\mathcal{V} \in \mathbb{R}^{n_1 \times d\times n_3 }} \frac{\operatorname{Trace}\left[\mathcal{V}^T \star \mathcal{A} \star \mathcal{V}\right]}{\operatorname{Trace}\left[\mathcal{V}^T  \star\mathcal{B} \star \mathcal{V}\right]},\text{ subject to}\ \ \mathcal{V}^T \star \mathcal{V}=\mathcal{I}_{d}.
\end{equation}
In this formulation, $\mathcal{V} \in \mathbb{R}^{n \times d\times n_3 }$ is required to have {f-orthonormal lateral slices}, i.e., $\mathcal{V}(:,i,:)^T\star \mathcal{V}(:,i,:)=\bm{e}$, with $\bm{e}\in \mathbb{R}^{1\times 1\times n_3}$ has zero components and $\bm{e}(1,1,1)=1$,  and $\mathcal{A}\in \mathbb{R}^{n\times n\times n_3}$ is an {f-symmetric} tensor, $\mathcal{B}\in \mathbb{R}^{n\times n\times n_3}$ is assumed to be {f-symmetric} and positive definite tensor.
This problem can be replaced by a simpler, yet not equivalent problem
\begin{equation}
	\max _{\mathcal{V} \in \mathbb{R}^{n \times d\times n_3 }} \operatorname{Trace}\left[\mathcal{V}^T \star \mathcal{A} \star \mathcal{V}\right], \text{ subject to    }
	\mathcal{V}^T \star \mathcal{B}  \star \mathcal{V}=\mathcal{I}_{d}.
\end{equation}
In practice, Problem \eqref{equ34} often arises as a simplification of an objective function that is more difficult to maximize, which can be described as follows
\begin{equation}
	\max _{\mathcal{V} \in \mathbb{R}^{n \times d\times n_3 }} \frac{\operatorname{Trace}\left[\mathcal{V}^T \star \mathcal{A} \star \mathcal{V}\right]}{\operatorname{Trace}\left[\mathcal{V}^T  \star\mathcal{B} \star \mathcal{V}\right]},\text{ subject to}\ \ \mathcal{V}^T \star \mathcal{C} \star \mathcal{V}=\mathcal{I}_{d}, 
\end{equation}
where $\mathcal{B}$ and $\mathcal{C}$ are assumed to be {f-symmetric} and positive definite for simplicity. The tensor $\mathcal{C}$ defines the desired {f-orthogonality} and in the simplest case, it is just the Identity tensor.
\subsection{Existence and uniqueness of a solution}
Let us consider the following theorem, which is important in our analysis.
\begin{theorem}\label{theo1}
Let $\mathcal{A} \in \mathbb{R}^{n \times n \times n_3} $ and $\mathcal{B}\in\mathbb{R}^{n \times n \times n_3} $, where $\mathcal{A}$ is {f-symmetric} and $\mathcal{B}$ is an {f-symmetric} positive definite tensor. Consider the optimization problem as follows
	\begin{equation}\label{equ31}
		\underset{{\mathcal{U}}\in \mathbb{R}^{n \times d \times n_3}}{\max } \operatorname{Trace}\left({\mathcal{U}}^T \star \mathcal{A} \star {\mathcal{U}}\right), \text{ subject to } {\mathcal{U}}^T \star \mathcal{B} \star {\mathcal{U}}={\mathcal{I}_{d}}.
	\end{equation}
	The problem \eqref{equ31} achieves a maximum, and the solution of \eqref{equ31} is the $d$ eigenslices associated to the $d$ largest eigentube of the following generalized eigentube problem
 	\begin{equation}\label{equ32}
		{ \mathcal{A} \star\mathcal{\mathcal{U}}=\mathcal{B}} \star \mathcal{U} \star {\Lambda}, 
	\end{equation}
 where $\Lambda\in \mathbb{R}^{d\times d\times n_3}$ is a f-diagonal tensor.\\
	 In the case of minimization, the solution of \eqref{equ31} is the $d$ eigenslices associated with the $d$ smallest non-zero eigentube of the generalized eigentube problem \eqref{equ32}.
\end{theorem}
\begin{proof}
	We have the set of tensors $\mathcal{U}$ such that $\mathcal{U}^T \star\mathcal{B}  \star\mathcal{U}=\mathcal{I}_{d}$ is closed, under the assumptions, the function in the right-hand side of \eqref{equ31} is a continuous function of its argument, therefore, the maximum of the problem \eqref{equ31} is reached. Then the Lagrange multipliers of \eqref{equ31} is given by
	\[
	\phi(\mathcal{U},\mathcal{L})=\operatorname{Trace}\left(\mathcal{U}^{T} \star \mathcal{A} \star \mathcal{U}\right)-\langle \mathcal{L} ,\mathcal{U}^{T} \star \mathcal{B} \star \mathcal{U}-\mathcal{I}_{d}\rangle,
	\]
	where $ \mathcal{L} $ represents the Lagrangian tensor.\\
	To compute the derivative of $\phi$ respect to $\mathcal{U}$ we need first to compute the derivative of $h_{\mathcal{A}}(\mathcal{U})=\operatorname{Trace}\left(\mathcal{U}^T \star \mathcal{A} \star \mathcal{U}\right) $. We have
 \[
 \langle \mathcal{A},\mathcal{B}\rangle = \text{Trace}\left(\mathcal{A}^T \star \mathcal{B}\right).
 \]
	Then $h_{\mathcal{A}}$ can be write as $h_{\mathcal{A}}(\mathcal{U}) = \langle \mathcal{A} \star \mathcal{U},\mathcal{U}\rangle$
	Therefore, the directional derivative of $h_{\mathcal{A}}$ with respect to $\mathcal{U}$ in the direction $\mathcal{H}$ can be expressed as
	\begin{align*}
		D_\mathcal{U}h_{\mathcal{A}} \left(\mathcal{H}\right)&=\lim _{t \rightarrow 0} \frac{h_{\mathcal{A}} (\mathcal{U} +t\mathcal{H}) - h_{\mathcal{A}} (\mathcal{U})}{t}\\
		&=\lim _{t \rightarrow 0} \frac{\langle \mathcal{A}\star(\mathcal{U} +t\mathcal{H}),\mathcal{U} +t\mathcal{H}\rangle - \langle \mathcal{A}\star\mathcal{U},\mathcal{U}\rangle}{t}\\
		&=\lim _{t \rightarrow 0} \frac{ t\langle\mathcal{A}\star\mathcal{U},\mathcal{H}\rangle +t\langle \mathcal{A}\star\mathcal{H},\mathcal{U}\rangle +t^{2}\langle\mathcal{A}\star\mathcal{H},\mathcal{H}\rangle}{t}\\
		&=\langle2\mathcal{A}\star\mathcal{U},\mathcal{H}\rangle.
	\end{align*}
	Therefor the derivative of $h_{\mathcal{A}}$ with respect to $\mathcal{U}$ is equal to $2\mathcal{A} \star  \mathcal{U}$. To find the optimum, we set the derivatives of the Lagrange function equal to zero, so we obtain
 {
	\begin{equation}\label{equ33}
		\frac{\partial \phi\left(\mathcal{U},\mathcal{L}\right)}{\partial \mathcal{U}} = 2 \mathcal{A} \star \mathcal{U} - 2\mathcal{B} \star \mathcal{U} \star \mathcal{L} = 0 \\
		\iff \mathcal{A} \star \mathcal{U} = \mathcal{B} \star \mathcal{U} \star \mathcal{L}.
	\end{equation}
 }
        We know that the optimum verify ${\mathcal{U}}^T \star \mathcal{B} \star {\mathcal{U}}={\mathcal{I}_{d}}$ therefor $\mathcal{U}^T \star \mathcal{A} \star \mathcal{U}=\mathcal{L}$ and since $\mathcal{A}$ is {f-symmetric}, then $\mathcal{L}$ is {f-symmetric}. Let's consider the {f-diagonalization} of $\mathcal{L}$, i.e., $\mathcal{L}=\mathcal{P}\star \Lambda \star \mathcal{P}^T$, with $\Lambda\in \mathbb{R}^{d\times d\times n_3}$ is an f-diagonal tensor. Thus, Equation \eqref{equ33} becomes 
        \begin{equation}\label{generalized}
        \mathcal{A}\star \mathcal{U}\star \mathcal{P}=\mathcal{B}\star \mathcal{U}\star \mathcal{P}\star \Lambda.
        \end{equation}
        Consequently, to solve the problem \eqref{equ34}, we must go through the above tensor generalized eigenproblem.\\
        {The generalized eigenvalue problem \eqref{generalized} admits $d$ real eigentube problem, because its eigenvalue are those of $\Lambda$, and $\Lambda$ is f-symmetric (f-diagonal tensor), so by referring to \cite{el2023spectral}, it admits $n$ real eigentubes. Then the generalized eigenproblem admits $d$ real eigentubes}.\\
        If these eigentubes are labeled decreasingly as has been shown before, and if $\mathcal{V}=\left[\mathcal{V}_1, \ldots, \mathcal{V}_d\right] \in \mathbb{R}^{n \times d\times n_3} $ is the set of eigenslices associated with the first $d$ eigentube with $\mathcal{V}^T \star \mathcal{B} \star\mathcal{V}=\mathcal{I}_{d}$, then we have
        {
	\[
	\max _{\substack{\mathcal{U} \in \mathbb{R}^{n \times d\times n_3}\\\mathcal{U}^{T} \star \mathcal{B} \star \mathcal{U} =\mathcal{I}_{d} }} \operatorname{Trace}\left[\mathcal{U}^T \star \mathcal{A} 
	\star \mathcal{U}\right]=\operatorname{Trace}\left[\mathcal{V}^T \star \mathcal{A} \star \mathcal{V}\right]= \frac{1}{n_3}\sum_{i=1}^{n_3} \sum_{j=1}^{d} \widehat{\lambda}^{(i)}_{j},
	\]
 }
	where $\lambda_i={\tt ifft}\left(\widehat{\lambda}_{i},[\,],3\right)$ of size ${1\times 1\times n_3}$ is the first $j$-th eigentube of the generalised eigentube problem \eqref{generalized}.
 In the case of minimization, $\mathcal{V}$ is the $d$ eigenslices associated with the $d$ smallest non-zero eigentube of the generalized eigentube problem \eqref{equ33}.
\end{proof}

It is helpful to examine the 
$\operatorname{Trace}\left[\mathcal{V}^T \star\mathcal{B} \star \mathcal{V} \right]$ in detail. Let $\mathcal{B}=\mathcal{Q} \star \Lambda_{\mathcal{B}} \star \mathcal{Q}^T$ the {f-diagonalization} of $\mathcal{B}$, {where $\mathcal{Q}$ is f-orthogonal and $\Lambda_\mathcal{B}$ is a f-diagonal tensor contains the eigenslices and eigentubes of $\mathcal{B}$, respectively}. Let $ \vec{\mathcal{V}}_1, \ldots,  \vec{\mathcal{V}}_{d}$ be the lateral slices of $\mathcal{V}$, and define $ \vec{\mathcal{U}}_j=\mathcal{Q} \star  \vec{\mathcal{V}}_j$. Then we have
{
\begin{equation}\label{equ37}
	\operatorname{Trace}\left[\mathcal{V}^T \star \mathcal{B} \star \mathcal{V}\right]=\frac{1}{n_3}\sum_{k=1}^{n_3} \sum_{j=1}^d \sum_{i=1}^n \widehat{\lambda}_i^{(k)} \widehat{\vec{\mathcal{U}}}_{i j k}^2= \frac{1}{n_3}\sum_{k=1}^{n_3}  \sum_{i=1}^n \widehat{\lambda}_i^{(k)} \sum_{j=1}^d\left\vert \widehat{\Vec{\mathcal{U}}}_{i j k}\right\vert^2,
\end{equation}
}
where $\lambda_i={\tt ifft}\left(\widehat{\lambda}_{i},[\,],3\right)$ of size ${1\times 1\times n_3}$ is the largest $i$-th eigentube of $\mathcal{B}$.\\
{We can see that when $\mathcal{B}$ is f-symmetric positive definite, the quantity  Trace $\left[ \mathcal{V}^T \star \mathcal{B}\star \mathcal{V} \right]$ is non-vanishing.}\\
The following lemma examines under which conditions $\operatorname{Trace}\left[\mathcal{V}^T \star \mathcal{B} \star \mathcal{V}\right]$ is nonzero in the situation when $\mathcal{B}$ is positive semi-definite.
\begin{lemma}\label{lemma31}
	Assume that $\mathcal{B}$ is positive semi-definite and let $d$ be the number of lateral slices of $\mathcal{V}$. If $\mathcal{B}$ has at most $d-1$ zero eigentube then $\operatorname{Trace}\left[\mathcal{V}^T \star \mathcal{B} \star \mathcal{V} \right]$ is nonzero for any f-orthogonal $\mathcal{V}$.
\end{lemma}

\begin{proof}
	Using the previous notation $\mathcal{U}=\left[\vec{\mathcal{U}}_1, \cdots, \vec{\mathcal{U}}_d\right]$, has at least one $d\times d\times n_3$ subtensor which is nonsingular, so it has at least $d$ lateral slices that have a nonzero norm. Then in the sum \eqref{equ37}, at least one of the $n-d+1$ nonzero eigentube $\widehat{\lambda}_i \not \equiv \mathbf{0}$ will coincide with one of these lateral slices norms, and this sum will be nonzero.
\end{proof}\\
Therefore, the problem is well-posed under the condition that the null space of $\mathcal{B}$ is of dimension less than $d$, i.e., that its tubal rank be at least $n-d+1$. In this case, the maximum is finite.

Another situation that leads to difficulties is when the two traces in the problem  \eqref{equ34} have a zero value for the same $\mathcal{V}$. This situation should be excluded from consideration as it leads to an indefinite ratio of $0 / 0$. For this we must assume that $\operatorname{Null}(\mathcal{A}) \cap \operatorname{Null}(\mathcal{B})=\{0\}$.

\begin{proposition}
	Let $\mathcal{A}, \mathcal{B}$ be two {f-symmetric} tensors and assume that $\mathcal{B}$ is semi-positive definite with tubal rank $>n-d$ and that $\operatorname{Null}(\mathcal{A}) \cap \operatorname{Null}(\mathcal{B})=\{0\}$. Then the ratio \eqref{equ34} admits a finite maximum (resp. minimum) value $\rho_*$. The maximum is reached for a certain $\mathcal{V}$ that is unique up to f-orthogonal transforms of the lateral slices.
\end{proposition}
\begin{proof}
	The set of tensors $\mathcal{V}$ such that $\mathcal{V}^T \star \mathcal{V}=\mathcal{I}_{d}$ is closed and, under the assumptions, the ratio trace function in the right-hand side of \eqref{equ34} is a continuous function of its argument. Therefore, using Lemma \ref{lemma31} the maximum of the trace ratio \eqref{equ34} is reached.
\end{proof}
\subsection{ Necessary conditions for optimality} In this section, we search for the necessary conditions for optimality for the optimization problem \eqref{equ34}. Assume we have the conditions mentioned in Lemma \ref{lemma31} and Proposition \ref{pro1} on $\mathcal{A}$ and $\mathcal{B}$.
Therefore, the problem \eqref{equ34} admits a maximum,  and the corresponding Lagrangian function can write as,
\begin{equation}
	L(\mathcal{V}, \Gamma)=\frac{\operatorname{Trace}\left[\mathcal{V}^T \star \mathcal{A} \star \mathcal{V}\right]}{\operatorname{Trace}\left[\mathcal{V}^T \star \mathcal{B} \star \mathcal{V}\right]}-\langle\Gamma,\mathcal{V}^T \star \mathcal{V}-\mathcal{I}_{d}\rangle ,
\end{equation}
where $ \Gamma $ represents the Lagrangian tensor.
Based on the Karush-Kuhn-Tucker (KKT) optimality conditions, as \eqref{equ34} has a global maximum $\mathcal{V}_*$, there exists a Lagrangian multiplier tensor $\Gamma_*$, satisfying the following conditions
\[
\frac{\partial L\left(\mathcal{V}_*, \Gamma_*\right)}{\partial \mathcal{V}}=0 \quad \text { with } \quad \mathcal{V}_*^T \star \mathcal{V}_*=\mathcal{I}_{d}.
\]
We have the derivative of $L$ with respect to $\mathcal{V}$
\[
\frac{\partial L(\mathcal{V}, \Gamma)}{\partial \mathcal{V}}=\frac{2 \operatorname{Trace}\left[\mathcal{V}^T \star  \mathcal{B} \star \mathcal{V}\right] \mathcal{A} \star \mathcal{V}-2 \operatorname{Trace}\left[\mathcal{V} \star \mathcal{A} \star \mathcal{V}\right] \mathcal{B} \star \mathcal{V}}{\left(\operatorname{Trace}\left[\mathcal{V}^T \star \mathcal{B} \star \mathcal{V}\right]\right)^2}-\mathcal{V} \star \left(\Gamma^T+\Gamma\right).
\]

Hence, the optimal solutions $\mathcal{V}_*$ and $\Gamma_*$ verifies
\begin{equation}\label{eq3.9}
	\left(\mathcal{A}-\rho_* \mathcal{B}\right) \star \mathcal{V}_*=\frac{\operatorname{Trace}\left[\mathcal{V_*}^T \star \mathcal{B} \star \mathcal{V_*}\right]}{2} \mathcal{V}_* \star\left(\Gamma_*^T+\Gamma_*\right), 
\end{equation}
where $\rho_* = \operatorname{Trace}\left[\mathcal{V_*}^T \star \mathcal{A}  \star \mathcal{V_*}\right] / \operatorname{Trace}\left[\mathcal{V_*}^T  \star\mathcal{B} \star \mathcal{V_*}\right] .$ \\
Since the tensor $\left(\Gamma_*^T+\Gamma_*\right)$ is f-symmetric, it is f-diagonalizable. Let $ \mathcal{Q}$ be the tensor which f- diagonalizes $\Gamma_*^T+\Gamma_*$
\[
\Gamma_*^T+\Gamma_*=\mathcal{Q} \star  \Sigma_* \star \mathcal{Q}^T, \quad \mathcal{Q}^T \star \mathcal{Q}=\mathcal{I}_{d}.
\]
Define $\mathcal{U}_*=\mathcal{V}_* \star \mathcal{Q}$. We have $\mathcal{U}_*^T \star \mathcal{U}_*=\mathcal{I}_{d}$ and we can rewrite \eqref{eq3.9} as
\begin{equation}\label{equ10}
	\left(\mathcal{A}-\rho_* \mathcal{B}\right)  \star \mathcal{U}_*=\mathcal{U}_* \star \Lambda_*, \quad \text { where } \quad \Lambda_*=\frac{\operatorname{Trace}\left[\mathcal{V_*}^T \star \mathcal{B} \star \mathcal{V_*}\right]}{2} \Sigma_* .
\end{equation}
Equation \eqref{equ10} is the necessary condition of the pair $\rho_*, \mathcal{U}_*$ for the problem \eqref{equ34}.

\subsection{Newton-QR algorithm} We begin with the understanding that a maximum value, denoted by $\rho_*$, is attained for a specific (although not unique) f-orthogonal tensor, represented as $\mathcal{V}_*$. Consequently, for any f-orthogonal $\mathcal{V}$, the following inequality holds
\begin{equation}
	\dfrac{\operatorname{Trace}\left[\mathcal{V}^T \star \mathcal{A} \star \mathcal{V}\right]}{\operatorname{Trace}\left[\mathcal{V}^T \star\mathcal{B} \star \mathcal{V}\right]} \leq \rho_*.   
\end{equation}
This leads to the expression

\[
\operatorname{Trace}\left[\mathcal{V}^T \star \mathcal{A} \star \mathcal{V}\right] - \rho_*\operatorname{Trace}\left[\mathcal{V}^T \star \mathcal{B} \star \mathcal{V}\right] \leq 0.
\]
Then we have 
\[
\operatorname{Trace}\left[\mathcal{V_*}^T \star (\mathcal{A}  - \rho_*\mathcal{B})  \star \mathcal{V_*}\right] = 0.
\]
Therefore, we have the following necessary condition for $\rho_*,\ \mathcal{V}_*$ to be optimal
\begin{equation}
	\max _{\mathcal{V}^T \star \mathcal{V}=\mathcal{I}_{d}} \operatorname{Trace}\left[\mathcal{V}^T \star  \left(\mathcal{A}-\rho_* \mathcal{B}\right)  \star\mathcal{V}\right]=\operatorname{Trace}\left[\mathcal{V}_*^T \star \left(\mathcal{A}-\rho_* \mathcal{B}\right) \star \mathcal{V}_*\right]=0.
\end{equation}
According to Theorem \ref{theo1}, we can determine the solution of the problem
\begin{equation}\label{3.13}
	\underset{{\mathcal{V}}}{\arg \max } \operatorname{Trace}\left[{\mathcal{V}}^T \star (\mathcal{A}-\rho \mathcal{B}) \star {\mathcal{V}}\right] \quad \text { subject to } {\mathcal{V}}^T \star  {\mathcal{V}}={\mathcal{I}_{d}}.
\end{equation}
Consider the function
\[
f(\rho)=\max _{\mathcal{V}^T  \star\mathcal{V}=\mathcal{I}_{d}} \operatorname{Trace}\left[\mathcal{V}^T  \star(\mathcal{A}-\rho \mathcal{B}) \star \mathcal{V}\right].
\]
The tensor $ \mathcal{V} $ that reaches the maximum of \eqref{3.13} is not unique because changing the lateral slices of $\mathcal{V}$ with an f-orthogonal transformation does not affect the trace. To select the optimal $\mathcal{V}$, we use Theorem \ref{theo1}, the optimal is the set of eigenslices of the tensor $\mathcal{A} - \rho \mathcal{B}$. We will denote the set of eigenslices  that attain the specified maximum as $\mathcal{V}(\rho)$.\\
Therefore, $f$ can be written as
\[
f(\rho)=\operatorname{Trace}\left[\mathcal{V(\rho)}^T \star \left(\mathcal{A}-\rho \mathcal{B}\right)  \star \mathcal{V(\rho)}\right], \text{\ with\ } \mathcal{V}(\rho)^T \star \mathcal{V}(\rho)=\mathcal{I}_{d}.
\]
Also from Theorem \ref{theo1}, $\mathcal{V}(\rho)$ diagonalizes $\mathcal{A}-\rho \mathcal{B}$ and verifies
\[
(\mathcal{A}-\rho \mathcal{B}) \star \mathcal{V}(\rho)=\mathcal{V}(\rho) \star \mathcal{D}(\rho),
\]
where $\mathcal{D}(\rho)$ is a f-diagonal tensor of size $d \times d \times n_3$.\\
From the equality $ \mathcal{V}(\rho)^T  \star  \mathcal{V}(\rho)=\mathcal{I}_{d}$, we can write
\[
\begin{aligned}
	\frac{d}{d \rho}\left[ \mathcal{V}(\rho)^T \star  \mathcal{V}(\rho)\right]&=\frac{d  \mathcal{V}(\rho)^T}{d \rho}\star  \mathcal{V}(\rho)+ \mathcal{V}(\rho)^T \star \frac{d  \mathcal{V}(\rho)}{d \rho}=0\\
	&\Rightarrow \frac{d  \mathcal{V}(\rho)^T}{d \rho}\star  \mathcal{V}(\rho)=- \mathcal{V}(\rho)^T \star \frac{d  \mathcal{V}(\rho)}{d \rho}  \\
	&\Rightarrow   \operatorname{Diag}\left[ \mathcal{V}(\rho)^T \star \frac{d  \mathcal{V}(\rho)}{d \rho}\right]=\mathbf{0}.
\end{aligned}
\]
Because the tensor $\dfrac{d  \mathcal{V}(\rho)^T}{d \rho} \star  \mathcal{V}(\rho)$ is anti-symmetry.\\
\begin{remark}
  A tensor $\mathcal{A}$ is anti-symmetric if $\mathcal{A} = -\mathcal{A}^T$, and $\operatorname{Diag}(\mathcal{A})=\mathbf{0}$ if the f-diagonal of each frontal slice $\mathcal{A}^{(i)}$ is equal to $0$.  
\end{remark}

Our objective is to calculate the derivative of the function $f(\rho)$. Determining the derivative of $f(\rho)$ is essential for deriving this particular expression
\begin{equation*}
	\begin{aligned}
		&\frac{d}{d \rho}\left[ \mathcal{V}(\rho)^T \star (\mathcal{A}-\rho \mathcal{B}) \star \mathcal{V}(\rho)\right] = \frac{d}{d \rho}\left[ \mathcal{V}(\rho)^T \star \mathcal{A} \star \mathcal{V}(\rho)\right] - \frac{d}{d \rho}\left[ \mathcal{V}(\rho)^T \star \rho \mathcal{B} \star \mathcal{V}(\rho)\right] \\
		& = \frac{d \mathcal{V}(\rho)^T}{d \rho} \star \mathcal{A} \star \mathcal{V}(\rho) + \mathcal{V}(\rho)^T \star \mathcal{A} \star \frac{d \mathcal{V}(\rho)}{d \rho} - \frac{d \mathcal{V}(\rho)^T}{d \rho} \star \rho \mathcal{B} \star \mathcal{V}(\rho) \\
            &- \mathcal{V}(\rho)^T \star \left[ \mathcal{B} \star \mathcal{V}(\rho) + \rho \mathcal{B} \star \frac{d \mathcal{V}(\rho)}{d \rho} \right] \\
		& = \frac{d \mathcal{V}(\rho)^T}{d \rho} \star [\mathcal{A} - \rho \mathcal{B}] \star \mathcal{V}(\rho) + \mathcal{V}(\rho)^T \star [\mathcal{A} - \rho \mathcal{B}] \star \frac{d \mathcal{V}(\rho)}{d \rho} - \mathcal{V}(\rho)^T \star \mathcal{B} \star \mathcal{V}(\rho) \\
		& = \frac{d \mathcal{V}(\rho)^T}{d \rho} \star \mathcal{V}(\rho) \star \mathcal{D}(\rho) + \mathcal{D}(\rho) \star \mathcal{V}(\rho)^T \star \frac{d \mathcal{V}(\rho)}{d \rho} - \mathcal{V}(\rho)^T \star \mathcal{B} \star \mathcal{V}(\rho) .
	\end{aligned}
\end{equation*}
Finally, we can express the final form of the derivative of $f(\rho)$.
\[
\begin{aligned}
	f^{'}(\rho) & =\operatorname{Trace}\left[\frac{d  \mathcal{V}(\rho)^T}{d \rho} \star  \mathcal{V}(\rho) \star \mathcal{D}(\rho)+\mathcal{D}(\rho)  \star \mathcal{V}(\rho)^T \star \frac{d  \mathcal{V}(\rho)}{d \rho}- \mathcal{V}(\rho)^T \star \mathcal{B} \star  \mathcal{V}(\rho)\right] \\
	& =2 \operatorname{Trace}\left[\{\mathcal{D}(\rho) \star  \mathcal{V}(\rho)^T \star \frac{d  \mathcal{V}(\rho)}{d \rho}\right]-\operatorname{Trace}\left[ \mathcal{V}(\rho)^T \star \mathcal{B}\star   \mathcal{V}(\rho)\right] \\
	& =-\operatorname{Trace}\left[ \mathcal{V}(\rho)^T \star \mathcal{B} \star   \mathcal{V}(\rho)\right].
\end{aligned}
\]
We aim to find a solution to the equation where $f(\rho) = 0$, and then identify $\mathcal{V}(\rho)$, to do this, we introduce Newton-QR algorithm. 
Using the expression of the differential of $f$, the form of Newton's method is as follows
\[
\begin{aligned}
	\rho_{new}  &= \rho - \frac{f(\rho)}{f^{'}(\rho)}\\
	& = \rho + \frac{\operatorname{Trace}\left[\mathcal{V(\rho)}^T \star \left(\mathcal{A}-\rho \mathcal{B}\right)  \star\mathcal{V(\rho)}\right]}{\operatorname{Trace}\left[\mathcal{V(\rho)}^T \star \mathcal{B} \star \mathcal{V(\rho)}\right]}\\
	&= \frac{\operatorname{Trace}\left[\mathcal{V(\rho)}^T  \star\left(\mathcal{A}\right) \star \mathcal{V(\rho)}\right]}{\operatorname{Trace}\left[\mathcal{V(\rho)}^T  \star\mathcal{B} \star \mathcal{V(\rho)}\right]}.  
\end{aligned}
\]
In the tensor Newton-QR algorithm, we use the tensor-QR algorithm (t-QR Algorithm) for tensors to compute the eigenslices associated with the dominant eigentubes of the tensor $\mathcal{A}-\rho\mathcal{B}$. For more details about the t-QR Algorithm see \cite{el2023spectral}.
\begin{algorithm}[H]
	\caption{Tensor Newton–QR algorithm}
	\label{alg:tensor_newton_qr}
	\begin{algorithmic}
		\STATE{\textbf{Input:} $\mathcal{A} \in \mathbb{R}^{n \times n \times n_3}, \ \mathcal{B} \in \mathbb{R}^{n \times n \times n_3},$}
		\STATE{ $d$: Dimension of projective space,}
		\STATE{$\epsilon$: Tolerance,}
		\STATE{M: Max iteration,}
		\STATE{$\mathcal{V}_0 \in \mathbb{R}^{n \times d \times n_3}$: Initial tensor,}
		\STATE{ Compute $\rho = \frac{\operatorname{Trace}[\mathcal{V}_0^T \star \mathcal{A} \star \mathcal{V}_0]}{\operatorname{Trace}[\mathcal{V}_0^T \star \mathcal{B} \star \mathcal{V}_0]}$,}\\
		\FOR{ $i = 1$ to $M$}
		\STATE{ $[\mathcal{V}(\rho)_d, \Lambda(\rho)_d] \leftarrow$ Select the $d$ eigenslices associated to $d$ largest eigentubes of $(\mathcal{A}-\rho \mathcal{B})$ using Tensor QR algorithm,} 
		\STATE{$\rho_{\text{new}} = \frac{\operatorname{Trace}[\mathcal{V}(\rho)_d^T \star \mathcal{A} \star \mathcal{V}(\rho)_d]}{\operatorname{Trace}[\mathcal{V}(\rho)_d^T \star \mathcal{B} \star \mathcal{V}(\rho)_d]}$,}
		\IF{$|\rho - \rho_{\text{new}}| \leq \epsilon$}
		\STATE{break,}
		\ENDIF
		\STATE{$\rho = \rho_{\text{new}}$,}
		\ENDFOR
		\RETURN {$\mathcal{V}(\rho)_d$,}
	\end{algorithmic}
        \label{algo4}
\end{algorithm}

\section{Trace-Ratio tensor methods}\label{s4} In this section, we introduce trace ratio methods for tensor dimensionality reduction. Firstly, we introduce the concept of graphs and their relation with tensors.
\subsection{Graphs and Multidimensional data}\label{s41} In this part, we introduce the concepts of graphs and their relation with high-order data. Additionally, we  demonstrate how to compute the affinity tensor \cite{henaff2015deep}.

Consider the third-order data represented by the tensor $\mathcal{X} \in \mathbb{R}^{n_1 \times n_2 \times n_3}$. This data contains $n_1$ samples $\left\{\left({\mathcal{X}_i}, y_i\right)\right\}_{i=1}^{n_1}$, where $\mathcal{X}_i \in \mathbb{R}^{1 \times n_2 \times n_3}$, and $y_i \in\{1,2, \cdots, c\}$, $c$ is the number of classes. 
Let ${G}_{r}=({C}_{r},{E}_{r})$ and ${G}_{r}^{'}=({C}^{'}_{r},{E}^{'}_{r})$ be two types of graphs both over the $r$-th frontal slice of $ \widehat{\mathcal{X}}={\tt fft}\left(\mathcal{X},[\,], 3\right)$.\\
To construct the graphs ${G}_{r}$, we consider each pair of points $\widehat{\mathcal{X}}^{(r)}_i$ and $\widehat{\mathcal{X}}^{(r)}_j$ from the same class, i.e., $y_i = y_j$ for the $r$-th frontal slice of $\widehat{\mathcal{X}}$. We link nodes ${C}_{r_i}$ and ${C}_{r_j}$ if $\widehat{\mathcal{X}}^{(r)}_i$ and $\widehat{\mathcal{X}}^{(r)}_j$ are close.\\
There are two variations
\begin{enumerate}
	\item $\epsilon$-neighborhood: connect ${C}_{r_i}$ and ${C}_{r_j}$ by an edge if $\left\|\widehat{\mathcal{X}}^{(r)}_{i}-\widehat{\mathcal{X}}^{(r)}_{j}\right\|_{F}^{2} \leq \epsilon$.
	\item $k$-nearest neighbors: connect ${C}_{r_i}$ and ${C}_{r_j}$ by an edge if $\widehat{\mathcal{X}}^{(r)}_i$ is among the $k$-nearest neighbors of $\widehat{\mathcal{X}}^{(r)}_j$ or $\widehat{\mathcal{X}}^{(r)}_j$ is among the $k$-nearest neighbors of $\widehat{\mathcal{X}}^{(r)}_i$.
\end{enumerate}
To construct the graphs ${G}^{'}_{r}$, we consider each pair of points $\widehat{\mathcal{X}}^{(r)}_i$ and $\widehat{\mathcal{X}}^{(r)}_j$ from different class, i.e. $y_i \neq y_j$ for the $r$-th frontal slice of $\widehat{\mathcal{X}}$. We link nodes ${C}^{'}_{r_i}$ and ${C}^{'}_{r_j}$ if $\widehat{\mathcal{X}}^{(r)}_i$ and $\widehat{\mathcal{X}}^{(r)}_j$ are close.
\subsubsection{Definition of the weights}
To define the affinity tensor $\widehat{\mathcal{W}} \in \mathbb{C}^{n_1 \times n_1 \times n_3}$ of ${G}_{r}$. There are two variations
\begin{enumerate}
	\item Heat kernel (parameter $t \in \mathbb{R}$ ): If nodes ${C}_{r_i}$ and ${C}_{r_j}$ are connected, put
	\[
	\widehat{\mathcal{W}}_{i j  r}=e^{-\dfrac{\left\|\widehat{\mathcal{X}}_i^{(r)}-\widehat{\mathcal{X}}_j^{(r)}\right\|_{F}^2}{t}} ;\  r = 1\  \ldots,\  n_3,\ i,j= 1\  \ldots,\  n_1.
	\]
	otherwise, put $\widehat{\mathcal{W}}_{i j r}=0$.
	\item Simple-minded (no parameters $(t=\infty)$ ):
	\[
	\widehat{\mathcal{W}}_{i j r}= \begin{cases} 1 & \text { if nodes ${C}_{r_i}$ and ${C}_{r_j}$ are connected by an edge, }  \\ 
     0 & \text { otherwise. }\end{cases}
	\]
	This simplification avoids the need to choose $t$.
\end{enumerate}
We use the same variations to calculate the affinity weight $\widehat{\mathcal{W}}^{'}$ of the graphs ${G}_{r}^{'}$.
\\
Note that the tensors ${\mathcal{W}}$ and ${\mathcal{W}^{'}}$ are f-symmetric.

\subsection{Multilinear Local Discriminant Embedding} This method (MLDE) is a supervised dimensionality reduction algorithm, which requires as inputs a data tensor $\mathcal{X} \in \mathbb{R}^{n_1 \times n_2\times n_3} $, where $n_2$ represents the number of data points, and each sample represented by a third-order tensor $\left\{\mathcal{X}_{i} \in \mathbb{R}^{n_1 \times 1 \times n_3 }\right\}$, the desired number of dimensions $d<n_1$, integers $k_1$, $k_2$ for finding local neighborhoods, and the output is a tensor $\mathcal{V} \in \mathbb{R}^{n_1 \times d\times n_3 }$, then the reduced data is obtained by $\mathcal{Y} =\mathcal{V}^{\top} \star \mathcal{X} \in \mathbb{R}^{d \times n_2 \times n_3 } $. This method can be divided into three main steps; construct the neighborhood graphs, compute the affinity weights, and complete the embedding. The key idea behind the third step of the MLDE algorithm is to minimize the distance between the neighboring points of the same class and at the same time maximize the distance between the neighboring points of different classes of each frontal slice in Fourier domain. By considering these two aspects, we get  the following optimization problem, for $ r = 1, \ldots, n_3$
\begin{equation}\label{equ4.1}
        \underset{\widehat{\mathcal{V}}^{(r)} \in \mathbb{C}^{n_1 \times d \times 1}}{\arg \max} \frac{\displaystyle\sum_{i,j=1}^{n_2} \widehat{\mathcal{W}}_{ijr}^{\prime} \left\| \widehat{\mathcal{V}}^{(r)\top} \widehat{\mathcal{X}}^{(r)}_{i} -  \widehat{\mathcal{V}}^{(r)\top} \widehat{\mathcal{X}}^{(r)}_{j} \right\|_{F}^{2} }{  \displaystyle\sum_{i,j=1}^{n_2} \widehat{\mathcal{W}}_{ijr}\left\| \widehat{\mathcal{V}}^{(r)\top} \widehat{\mathcal{X}}^{(r)}_{i} -\widehat{\mathcal{V}}^{(r)\top} \widehat{\mathcal{X}}^{(r)}_{j} \right\|_{F}^{2} }, \text{ subject to } \widehat{\mathcal{V}}^{(r)\top}\widehat{\mathcal{V}}^{(r)}=\widehat{\mathcal{I}}^{(r)}_{d},
\end{equation}
where $\widehat{\mathcal{W}}_{i j l}$ and $\widehat{\mathcal{W}}_{i j l}^{\prime}$ are the elements of the affinity tensors. We compute these affinity tensors using the notions described in Section \ref{s41}.\\
To gain more insight into \eqref{equ4.1}, we write the square of the norm in the form of a trace
\[
\begin{aligned}
        f_r(\widehat{\mathcal{V}}^{(r)}) &= \sum_{i, j =1}^{n_2} \left\| \widehat{\mathcal{V}}^{(r)\top} \widehat{\mathcal{X}}_{i}^{(r)} - \widehat{\mathcal{V}}^{(r)\top} \widehat{\mathcal{X}}_{j}^{(r)} \right\|_{F}^{2} \widehat{\mathcal{W}}_{i j r}^{\prime} \\
        &= \displaystyle\sum_{i,j = 1}^{n_2} \operatorname{Trace}\left( \left( \widehat{\mathcal{V}}^{(r)\top} \widehat{\mathcal{X}}_{i}^{(r)} - \widehat{\mathcal{V}}^{(r)\top} \widehat{\mathcal{X}}_{j}^{(r)} \right) \left( \widehat{\mathcal{V}}^{(r)\top} \widehat{\mathcal{X}}_{i}^{(r)} - \widehat{\mathcal{V}}^{(r)\top} \widehat{\mathcal{X}}_{j}^{(r)} \right)^{\top} \right) \widehat{\mathcal{W}}_{i j r}^{\prime} \\
        &= \operatorname{Trace}\left( \widehat{\mathcal{V}}^{(r)\top} \displaystyle\sum_{i, j = 1}^{n_2} \left( \widehat{\mathcal{X}}_{i}^{(r)} - \widehat{\mathcal{X}}_{j}^{(r)} \right) \left( \widehat{\mathcal{X}}_{i}^{(r)} - \widehat{\mathcal{X}}_{j}^{(r)} \right)^{\top} \widehat{\mathcal{W}}_{i j r}^{\prime} \widehat{\mathcal{V}}^{(r)} \right) \\
        &= \operatorname{Trace}\left( \widehat{\mathcal{V}}^{(r)\top} \left( 2 \widehat{\mathcal{X}}^{(r)} \widehat{\mathcal{D}}^{(r)\prime} \widehat{\mathcal{X}}^{(r)\top} - 2 \widehat{\mathcal{X}}^{(r)} \widehat{\mathcal{W}}^{(r)\prime} \widehat{\mathcal{X}}^{(r)\top} \right) \widehat{\mathcal{V}}^{(r)} \right) \\
        &= 2 \operatorname{Trace}\left( \widehat{\mathcal{V}}^{(r)\top} \widehat{\mathcal{X}}^{(r)} \left( \widehat{\mathcal{D}}^{(r)\prime} - \widehat{\mathcal{W}}^{(r)\prime} \right) \widehat{\mathcal{X}}^{(r)\top} \widehat{\mathcal{V}}^{(r)} \right),
\end{aligned}
\]

where $\widehat{\mathcal{X}} = \texttt{fft}(\mathcal{X},[\,],3)$, with $\mathcal{X}$ is the data tensor, and $\widehat{\mathcal{D}}^{\prime}$ is a f-diagonal tensor with $\widehat{\mathcal{D}}^{\prime}_{i i l} = \displaystyle\sum_{j =1}^{n_2} \widehat{\mathcal{W}}^{\prime}_{i j l}$. 
We set
\[
F({\mathcal{V}}) = \sum_{i=1}^{n_3} f_r\left(\widehat{\mathcal{V}}^{(r)}\right),
\]
then, by using tensor notation and the definition of trace, we can write $F$ as 
    \[
    F(\mathcal{V}) = 
 2 n_3 \operatorname{Trace}\left(\mathcal{V}^{\top} \star  \mathcal{X} \star \left( \mathcal{D}^{\prime} - \mathcal{W}^{\prime} \right) \star \mathcal{X}^{\top} \star \mathcal{V}\right).
    \]

For the denominator of the objective function in \eqref{equ4.1}, we use the same analogy as before

    \[
    H(\mathcal{V}) = 2 n_3 \operatorname{Trace}\left(\mathcal{V}^{\top} \star  \mathcal{X} \star \left( \mathcal{D} - \mathcal{W} \right) \star \mathcal{X}^{\top} \star \mathcal{V}\right).
    \]

Then, the problem \eqref{equ4.1} is equivalent to the following trace ratio tensor problem
\begin{equation}
	\begin{gathered}
		\underset{\mathcal{V}\in \mathbb{R}^{n_2 \times d \times n_3 } }{\arg \max } \frac{\operatorname{Trace}\left(\mathcal{V}^T  \star \mathcal{X} \star \left( \mathcal{D}^{\prime} - \mathcal{W}^{\prime} \right) \star \mathcal{X}^{\top}  \star\mathcal{V}\right) }{\operatorname{Trace}\left(\mathcal{V}^{\top} \star  \mathcal{X} \star \left( \mathcal{D} - \mathcal{W} \right) \star \mathcal{X}^{\top} \star \mathcal{V}\right) }, \text{ subject to } \mathcal{V}^{\top}\star\mathcal{V} = \mathcal{I}_{d}.
	\end{gathered}
\end{equation}
This optimization problem can be simplified as
\begin{equation}\label{equ4.3}
	\begin{gathered}
		\underset{\mathcal{V} \in \mathbb{R}^{n_2 \times d\times n_3 } }{\arg \max } \frac{\operatorname{Trace}\left( \mathcal{V}^{\top}  \star\mathcal{L^{\prime}} \star \mathcal{V} \right) }{\operatorname{Trace}\left( \mathcal{V}^{\top} \star \mathcal{L} \star \mathcal{V} \right)}, \text{ subject to } \mathcal{V}^{\top}\star\mathcal{V} = \mathcal{I}_{d},\\
	\end{gathered}
\end{equation}
{
Consider the expressions $\mathcal{L}^{'} = \mathcal{X} \star (\mathcal{D}^{\prime} - \mathcal{W}^{\prime}) \star \mathcal{X}^{\top}$ and $\mathcal{L} = \mathcal{X} \star (\mathcal{D} - \mathcal{W}) \star \mathcal{X}^{\top}$. Here, both $\mathcal{D}^{\prime} - \mathcal{W}^{\prime}$ and $\mathcal{D} - \mathcal{W}$ represent Laplacian tensors.\\
We have the tensor $\mathcal{L}^{'}$ f-symmetric, because $ \mathcal{D}^{\prime} - \mathcal{W}^{\prime} $ is f-symmetric then $ \mathcal{X} \star (\mathcal{D}^{\prime} - \mathcal{W}^{\prime}) \star \mathcal{X}^{\top}$ is f-symmetric,  and the tensor $\mathcal{D} - \mathcal{W}$ is the Laplacian tensor then is f-symmetric positive semi-definite then the tensor $\mathcal{L} =  \mathcal{X} \star \left( \mathcal{D} - \mathcal{W} \right) \star \mathcal{X}^{\top}$ is a f-symmetric 
positive semi-define tensor. We examine this type of problem in Section \ref{s3}, we use the Newton-QR Tensor Algorithm \ref{algo4} to get the solution for this trace ratio problem.}
Algorithm \ref{MLDE} shows a summary of the Multilinear Local Discriminant Embedding Algorithm method.
\begin{algorithm}[h]
	\caption{Multilinear Local Discriminant Embedding Algorithm}
	\label{MLDE}
		{\textbf{Input:} $\mathcal{X} \in \mathbb{R}^{n_1 \times n_2 \times n_3}, $ input data: third-order tensor.}\\
		\hspace*{1.2 cm}{ $Y$: labels: $c$ classes.}\\
		\hspace*{1.3 cm}{$d$: reduced dimension.}\\
		\hspace*{1.3 cm}{$k_1,\ k_2$: number of neighborhoods.}\\
        \hspace*{1.3 cm}{$\epsilon$: Tolerance.}\\
		\hspace*{1.25 cm}{$M$: Max iteration.}\\
		\textbf{Initialize: }{$\mathcal{V}_0 \in \mathbb{R}^{n_1 \times d \times n_3}$.}\\
		{\textbf{Output:} $\mathcal{V}\in \mathbb{R}^{n_1 \times d \times n_3}$,}
		\begin{algorithmic}[1]
		\STATE{Compute $\widehat{\mathcal{X}}={\tt fft}(\mathcal{X},[\,], 3)$,}
		\FOR{$i=1, \ldots,n_3$}
            \STATE{ From $ \widehat{\mathcal{X}}^{(i)} $ and $Y$ build two graphs $G_r$ and $G_r^{'}$. see sub-section \ref{s41} }
  		\STATE{$\widehat{\mathcal{D}}^{(i)} - \widehat{\mathcal{W}}^{(i)}, \widehat{\mathcal{D}}^{'(i)} -\widehat{\mathcal{W}}^{'(i)}\leftarrow$ Buildlaplacian  $\left(\widehat{\mathcal{X}}^{(i)}, Y, k_1, k_2\right)$,}
		\ENDFOR
  \STATE{ ${\mathcal{D}}={\tt ifft}(\widehat{\mathcal{D}},[\,], 3)$, \  ${\mathcal{W}}={\tt ifft}(\widehat{\mathcal{W}},[\,], 3)$, }

 \STATE{ ${\mathcal{D}}^{\prime}={\tt ifft}(\widehat{\mathcal{D}}^{\prime},[\,], 3)$, \  ${\mathcal{W}}^{\prime}={\tt ifft}(\widehat{\mathcal{W}}^{\prime},[\,], 3)$, }
  
        \STATE{ $\mathcal{L} =  \mathcal{X} \star (\mathcal{D} - \mathcal{W}) \star \mathcal{X}^{\top} $,\\
        $\mathcal{L}^{\prime} =  \mathcal{X} \star (\mathcal{D}^{\prime} - \mathcal{W}^{\prime}) \star \mathcal{X}^{\top} $ ,}
        \STATE{$\mathcal{V} \leftarrow $ Tensor Newton–QR algorithm $\left(  \mathcal{L}^{\prime}, \mathcal{L}, d, \epsilon, M  , \mathcal{V}_{0}\right)$, \ see Algorithm \ref{alg:tensor_newton_qr}.}
	\end{algorithmic}
\end{algorithm}

\subsection{Multilinear Laplacian Eigenmaps} This method (MLE) is an unsupervised non-linear dimensionality reduction algorithm, requires as inputs, data tensor $\mathcal{X} \in \mathbb{R}^{n_1\times n_2 \times n_3} $, with each sample represented by a third-order tensor $\left\{\mathcal{X}_{i} \in \mathbb{R}^{1 \times n_{2} \times n_{3}}, i=1, \ldots, n_{1}\right\}$, a dimension $d<n_2$ and integer $k$ for finding local neighborhoods. The output is $\mathcal{Y} \in \mathbb{R}^{n_{1} \times d \times n_{3}}$. The key idea behind this method is to minimize the distance between the neighboring points in low dimensional space for each frontal slice in Fourier domain. With this in mind, we minimize the following function. 
\begin{equation}\label{equ4.4}
	g_{r}(\widehat{\mathcal{Y}}^{(r)})=\frac{1}{2} \sum_{i=1}^{n_1} \sum_{j=1}^{n_1} \widehat{\mathcal{W}}_{i j l}\left\|\widehat{\mathcal{Y}}^{(r)}_{i}-\widehat{\mathcal{Y}}_{j}^{(r)}\right\|^2_{F},\   r = 1,2, \ldots ,n_3,  
\end{equation}
with $\widehat{\mathcal{W}}_{i j l}$ are the elements of the affinity tensor, see sub-Section \ref{s41}. The objective function \eqref{equ4.4} can be written as 
\[
\begin{aligned}
    g_{r}(\widehat{\mathcal{Y}}^{(r)})&=\frac{1}{2} \sum_{i=1}^{n_1} \sum_{j=1}^{n_1} \widehat{\mathcal{W}}_{i jr}\left\|\widehat{\mathcal{Y}}_{i}^{(r)}-\widehat{\mathcal{Y}}_{j}^{(r)}\right\|^2_{F}\\
    &=\frac{1}{2} \sum_{i=1}^{n_1} \sum_{j=1}^{n_1} \widehat{\mathcal{W}}_{i jr}\left(\widehat{\mathcal{Y}}_{i}^{(r)}-\widehat{\mathcal{Y}}_{j}^{(r)}\right) \left(\widehat{\mathcal{Y}}_{i}^{(r)\top}-\widehat{\mathcal{Y}}_{j}^{(r)}\right)^{\top} \\
    &= \frac{1}{2} \sum_{i=1}^{n_1} \sum_{j=1}^{n_1} \left( \widehat{\mathcal{W}}_{ijr} \widehat{\mathcal{Y}}_{i}^{(r)} \widehat{\mathcal{Y}}_{i}^{(r)\top} - \widehat{\mathcal{W}}_{ijr} \widehat{\mathcal{Y}}_{i}^{(r)} \widehat{\mathcal{Y}}_{j}^{(r)\top} - \widehat{\mathcal{W}}_{ijr} \widehat{\mathcal{Y}}_{j}^{(r)} \widehat{\mathcal{Y}}_{i}^{(r)\top} + \widehat{\mathcal{W}}_{ijr} \widehat{\mathcal{Y}}_{j}^{(r)} \widehat{\mathcal{Y}}_{j}^{(r)\top} \right) \\
    &= \frac{1}{2} \left[ \sum_{i=1}^{n_1} \widehat{\mathcal{D}}_{iir} \widehat{\mathcal{Y}}_{i}^{(r)} \widehat{\mathcal{Y}}_{i}^{(r)\top} - 2 \sum_{i=1}^{n_1} \sum_{j=1}^{n_1} \widehat{\mathcal{W}}_{ijr} \widehat{\mathcal{Y}}_{i}^{(r)} \widehat{\mathcal{Y}}_{j}^{(r)\top} + \sum_{j=1}^{n_1} \widehat{\mathcal{D}}_{jjr} \widehat{\mathcal{Y}}_{j}^{(r)} \widehat{\mathcal{Y}}_{j}^{(r)\top} \right].
\end{aligned}
\]

We can simplify it to 
\[
g_{r}(\widehat{\mathcal{Y}}^{(r)}) = \sum_{i=1}^{n_1} \widehat{\mathcal{D}}_{iir} \widehat{\mathcal{Y}}_{i}^{(r)} \widehat{\mathcal{Y}}_{i}^{(r)\top} - \sum_{i=1}^{n_1}\sum_{j=1}^{n_1}\widehat{\mathcal{W}}_{ijr}^{\prime} \widehat{\mathcal{Y}}_{i}^{(r)} \widehat{\mathcal{Y}}_{j}^{(r)\top}.
\]
We put 
\[
G(\mathcal{Y}) = \sum_{r = 1}^{n3} g_{r}\left( \widehat{\mathcal{Y}}^{(r)} \right).
\]
Then we can rewrite the equation using a tensor notation as 
\[
\begin{aligned}
	G(\mathcal{Y}) =n_3 \operatorname{Trace}\left(\mathcal{Y}^T \star (\mathcal{D}-\mathcal{W}) \star \mathcal{Y}\right)= n_3 \operatorname{Trace}\left(\mathcal{Y}^T \star \mathcal{L} \star \mathcal{Y}\right),
\end{aligned}
\]
with $\mathcal{L} =\mathcal{D}-\mathcal{W} $ and $\mathcal{D}$ is an f-diagonal tensor with $\widehat{\mathcal{D}}_{i i r}=\displaystyle\sum_{j= 1}^{n_1} \widehat{\mathcal{W}}_{j ir}$, and $\mathcal{Y}\in \mathbb{R}^{n_1 \times d \times n_3}$ is the tensor of the coordinates for the $n_1$ points.\\
Thus, the problem \eqref{equ4.4} is equivalent to the following constrained optimization problem 
\begin{equation}\label{equ4.5}
	\underset{\mathcal{Y}}{\arg \min }\  n_3 \operatorname{Trace}\left(\mathcal{Y}^T \star \mathcal{L} \star \mathcal{Y}\right),\quad \text {subject to } \mathcal{Y}^T \star\mathcal{D} \star \mathcal{Y}=\mathcal{I}_{d}.
\end{equation}
{
We have the tensor $\mathcal{L} = \mathcal{D} - \mathcal{W} $ is f-symmetric, and the tensor $\mathcal{D}$ is a f-diagonal tensor with all the elements of the diagonal is strictly positive then the tensor $\mathcal{D}$ is a f-symmetric positive definite then using the Theorem \ref{theo1}, the optimization problem is a generalized eigentube problem that is equivalent to}
\begin{equation}\label{equ4.6}
	\mathcal{L} \star \mathcal{Y}= \frac{1}{n_3}\mathcal{D} \star \mathcal{Y} \star \Lambda.  
\end{equation}
With $\Lambda$ is a f-diagonal tensor, the solution of \eqref{equ4.5} is the $d$ eigenslices associated with the $d$ smallest non-zero eigentube of the generalized eigentube problem \eqref{equ4.6}.
Algorithm \ref{MLE} shows a summary of the Multilinear Laplacian Eigenmaps method.
\begin{algorithm}[H]
	\caption{Multilinear Laplacian Eigenmaps}
	\label{MLE}
	\begin{algorithmic}
		\STATE{\textbf{Input:} $\mathcal{X} \in \mathbb{R}^{n_1 \times n_2 \times n_3}, $ input data: third-order tensor,}
		\STATE{$d$: reduced dimension,}
		\STATE{$k$: number of neighborhoods,}

		\STATE{\textbf{Output:} $\mathcal{Y}\in \mathbb{R}^{n_1 \times d \times n_3}$,}\\
		\STATE{Compute $\widehat{\mathcal{X}}={\tt fft}(\mathcal{X},[\,], 3)$,}
		\FOR{$i=1, \ldots,n_3$}
            \STATE{ From $ \widehat{\mathcal{X}}^{(i)} $ build a k-NN graph, }
  		\STATE{$\widehat{\mathcal{D}}^{(i)} - \widehat{\mathcal{W}}^{(i)}\leftarrow$ Buildlaplacian  $\left(\widehat{\mathcal{X}}^{(i)}, k\right)$,}
		\ENDFOR
        \STATE{ ${\mathcal{D}}={\tt ifft}(\widehat{\mathcal{D}},[\,], 3)$, }
        \STATE{ ${\mathcal{W}}={\tt ifft}(\widehat{\mathcal{W}},[\,], 3)$, }
        \STATE{ $\mathcal{L} = (\mathcal{D} - \mathcal{W}) $,}
        \STATE{$\mathcal{Y} \leftarrow $ Tensor diagonalization$\left(  \mathcal{D}^{-1} \star \mathcal{L}, d\right)$,}

		\STATE{\textbf{return} $\mathcal{Y}$}.
	\end{algorithmic}
\end{algorithm}

\subsection{Locally Multilinear Embedding} This method (LME) is an unsupervised non-linear dimensionality reduction algorithm, requiring as inputs, data tensor $\mathcal{X} \in \mathbb{R}^{n_1 \times n_2 \times n_3}$, with each sample represented by a third-order tensor $\left\{\mathcal{X}_{i} \in \mathbb{R}^{1 \times n_{2} \times n_{3}}, i=1, \ldots, n_{1}\right\}$, several dimensions $d<n_2$ and integer $k$ for finding local neighborhoods. The output is a tensor $\mathcal{Y} \in \mathbb{R}^{n_1 \times d \times n_3}$.
The main idea of local Multilinear Embedding is to use the same reconstruction weights in the lower-dimensional integration space as in the higher-dimensional input space. In the following sub-sections, we will explain this.
\subsubsection{Multilinear Reconstruction by the Neighbors}
In this section, we find the weights for the Multilinear reconstruction of every point by its $k$-NN. The optimization problem for this Multilinear reconstruction in the high-dimensional input space is given by the minimization of the function
\begin{eqnarray}
	\varepsilon({\mathcal{E}^{(r)}})&:= \displaystyle\sum_{i=1}^{n_1}\left\|{\mathcal{X}}^{(r)}_{i}- \displaystyle \sum_{j=1}^k \mathcal{E}^{(r)}_{i j} {\mathcal{X}}^{(r)}_{ij}\right\|_F^2,\nonumber\\
\text { subject to }& \displaystyle \sum_{j=1}^k \mathcal{E}^{(r)}_{i j}=1, \quad i  =1, \ldots, n_1, \quad r =1, \ldots, n_3,
\end{eqnarray}
where $\mathcal{X} \in \mathbb{R}^{n_1\times n_2 \times n_3}$ ${\mathcal{E}} \in \mathbb{R}^{n_1 \times k \times n_3}$, with
${\mathcal{E}}^{(r)}_{i} = [{\mathcal{E}}^{(r)}_{ik}, \ldots, {\mathcal{E}}^{(r)}_{ik}]^{T} \in\mathbb{R}^k$ includes the weights of Multilinear reconstruction of the $i$-th data point using its $k$ neighbors in the $r$-th frontal slice, and ${\mathcal{X}}^{(r)}_{ij}\in\mathbb{R}^{n_2}$ is the $j$-th neighbor of the $i$-th data point in the $r$-th frontal slices. The constraint $\sum_{j=1}^k \mathcal{E}^{(r)}_{i j}=1$ means that the weights of linear reconstruction sum to one for every point in each frontal slice.\\
We can write the objective $\varepsilon({\mathcal{E}^{(r)}})$ as
\[
\varepsilon({\mathcal{E}^{(r)}})=\sum_{i=1}^{n_1}\left\|{\mathcal{X}}_{i}^{(r)}-{\boldsymbol{X}}^{(r)}_{i}{\mathcal{E}}^{(r)}_{i}\right\|_F^2, \quad   r = 1, \ldots, n_3,
\]
with ${\boldsymbol{X}}^{(r)}_{i}\in \mathbb{R}^{n_2 \times k}$ contain the $k$ neighbor of the $i$-th data point in the $r$-th frontal slice.
The constraint $\sum_{j=1}^k \mathcal{E}^{(r)}_{ij}=1$ implies that $\mathbf{1}^{\top} {\mathcal{E}}^{(r)}_{i}=1$, where $ \mathbf{1}:=$ $[1, \ldots, 1]^{\top} \in \mathbb{N}^k$ therefore, ${\mathcal{X}}^{(r)}_{i}={\mathcal{X}}^{(r)}_{i} \mathbf{1}^{\top} \mathcal{E}^{(r)}_{i}$.\\
We can simplify the term in $\varepsilon(\mathcal{E}^{(r)})$ as
\[
\begin{aligned}
	\left \| \mathcal{X}^{(r)}_{i} - \boldsymbol{X}_{i}^{(r)} \mathcal{E}^{(r)}_{i} \right \|_F^2 &= \left \| \mathcal{X}_{i}^{(r)} \mathbf{1}^{\top} \mathcal{E}^{(r)}_{i} - \boldsymbol{X}_{i}^{(r)} \mathcal{E}_{i}^{(r)} \right \|_F^2 \\
	&= \left\| \left( \mathcal{X}_{i}^{(r)} \mathbf{1}^{\top} - \boldsymbol{X}_{i}^{(r)} \right) \mathcal{E}^{(r)}_{i} \right\|_F^2 \\
	&=  \mathcal{E}^{(r)\top}_{i} \left( \mathcal{X}^{(r)}_{i} \mathbf{1}^{\top} - \boldsymbol{X}^{(r)}_{i} \right)^{\top} \left( \mathcal{X}^{(r)}_{i} \mathbf{1}^{\top} - \boldsymbol{X}^{(r)}_{i} \right) \mathcal{E}_{i}^{(r)} \\
	&=  \mathcal{E}^{(r)\top}_{i} \boldsymbol{G}_{ir} \mathcal{E}_{i}^{(r)},
\end{aligned}
\]
where $\boldsymbol{G}_{ir}$ is a gram matrix defined as
\[
\boldsymbol{G}_{ir}:=\left({\mathcal{X}}_{i}^{(r)} \mathbf{1}^{\top}-{\boldsymbol{X}}_{i}^{(r)}\right)^{\top}\left({\mathcal{X}}_{i}^{(r)} \mathbf{1}^{\top}-{\boldsymbol{X}}_{i}^{(r)}\right)\in \mathbb{R}^{k \times k}, \   r = 1,\ldots, n_3.
\]
Finally, we have
\[
\begin{aligned}
	&\underset{\mathcal{E}_{i}^{(r)}}{\arg\min} \quad \sum_{i=1}^{n_1} \mathcal{E}^{(r) \top}_{i} \boldsymbol{G}_{ir} \mathcal{E}_{i}^{(r)}, \quad \text{for } i = 1, \ldots, n_1, r = 1, \ldots, n_3, \\
	&\text{subject to} \quad \mathbf{1}^{\top} \mathcal{E}_{i}^{(r)} = 1.
\end{aligned}
\]
The Lagrangian of this problem can be formulated as
{
\[
\mathcal{L}\left(\mathcal{E}^{(r)}_{i},\lambda_{ir} \right) = \sum_{i=1}^{n_1} \mathcal{E}^{(r) \top}_{i} \boldsymbol{G}_{ir} \mathcal{E}^{(r)}_{i} - \sum_{i=1}^{n_1} \lambda_{ir} \left(\mathbf{1}^{\top} \mathcal{E}_{i}^{(r)} - 1\right),
\]
setting the derivative of Lagrangian to zero gives
\[
\begin{aligned}
	\frac{\partial \mathcal{L}\left(\mathcal{E}^{(r)}_{i},\lambda_{ir} \right)}{\partial \mathcal{E}_{i}^{(r)}} & =2 \boldsymbol{G}_{ir} \mathcal{E}_{i}^{(r)}-\lambda_{ir} \mathbf{1} = \mathbf{0}, \   r = 1, ..., n_3,\\
	&\Longrightarrow \mathcal{E}_{i}^{(r)}=\frac{1}{2} {\boldsymbol{G}_{ir}}^{-1} \lambda_{ir} \mathbf{1}=\frac{\lambda_{ir}}{2} \boldsymbol{G}_{ir}^{-1} \mathbf{1}.\\
	\frac{\partial \mathcal{L}\left(\mathcal{E}^{(r)}_{i},\lambda_{ir} \right)}{\partial \lambda_{ir}} & =\mathbf{1}^{\top} \mathcal{E}^{(r)}_{i}-1 = 0 \Longrightarrow \mathbf{1}^{\top} \mathcal{E}^{(r)}_{i}=1,
\end{aligned}\\
\]
then we have
\[
\frac{\lambda_{ir}}{2} \mathbf{1}^{\top} \boldsymbol{G}_{ir}^{-1} \mathbf{1}=1 \Longrightarrow \lambda_{ir}=\frac{2}{\mathbf{1}^{\top} \boldsymbol{G}_{ir}^{-1} \mathbf{1}}, \    r = 1, \ldots,n_3,
\]
therefore
\[
\mathcal{E}^{(r)}_{i}=\frac{\lambda_{ir}}{2} \boldsymbol{G}_{ir}^{-1} \mathbf{1}=\frac{\boldsymbol{G}_{ir}^{-1} \mathbf{1}}{\mathbf{1}^{\top} \boldsymbol{G}_{ir}^{-1} \mathbf{1}} ,   r = 1,\ldots,n_3, \   i=1,\ldots,n_1.
\]
}
Moreover, reader must note that the rank of the matrix $\boldsymbol{G}_{ir}$, so the rank of matrix $\boldsymbol{G}_{ir} \in \mathbb{R}^{k \times k}$ is at most equal to $\min (k, n_2)$. If $n_2<k$, then $\boldsymbol{G}_{ir}$ is singular then $\boldsymbol{G}_{ir}$ should be replaced by $\boldsymbol{G}_{ir}+\epsilon \boldsymbol{I}$ where $\epsilon$ is a small positive number. Usually, the data such as images are high dimensional (so $k \ll n_2$ ) and thus if $\boldsymbol{G}_{ir}$ is full rank, we will not have any problem with inverting it.
\subsubsection{Multilinear Embedding}
In the last sub-section, we found the weights for Multilinear reconstruction in the high dimensional input space. In this sub-section, data points are projected in the low dimensional embedding space using the same weights as in the input space. This Multilinear embedding can be formulated as
\begin{eqnarray}\label{equ4.8}
	\underset{\widehat{\mathcal{Y}}^{(r)}}{\arg\min}&\displaystyle\sum_{i=1}^{n_1} \left\| \widehat{\mathcal{Y}}_{i}^{(r)} - \displaystyle\sum_{j=1}^{n_1} \widehat{\mathcal{W}}^{(r)}_{ij} \widehat{\mathcal{Y}}^{(r)}_{j} \right\|_F^2, \quad r = 1, \ldots, n_3,\nonumber \\
\text{subject to}& \quad \dfrac{1}{n_1} \displaystyle\sum_{i=1}^{n_1} \widehat{\mathcal{Y}}_{i}^{(r)} \widehat{\mathcal{Y}}_{i}^{(r)\top} = {{I}}, \quad \displaystyle\sum_{i=1}^{n_1} \widehat{\mathcal{Y}}_{i}^{(r)} = \mathbf{0},
\end{eqnarray}
where ${I}$ is the Identity matrix, and $\widehat{\mathcal{Y}}_{i} \in \mathbb{C}^{1 \times d\times n_3}$ is the $i$-th embedded data point in the Fourier domain, and $\widehat{\mathcal{W}}^{(r)}_{i j}$ is given by
\[
\widehat{\mathcal{W}}_{ij}^{(r)}:= \begin{cases}\widehat{\mathcal{E}}^{(r)}_{i j} & \text { if } {\mathcal{X}}_{j}^{(r)}\in \text{k-NN}\left({\mathcal{X}}_{i}^{(r)}\right),\ r= 1,\ldots ,n_3, \\ 0 & \text { otherwise.}\end{cases}
\]
The second constraint in equation \eqref{equ4.8} states that the mean of the projected data points is zero. The first and second constraints together ensure that the projected points have unit covariance.
We have $ \widehat{\mathcal{W}}_{i}^{(r)}\in \mathbb{C}^n$ and let $ e_i:=$ $[0, \ldots, 1, \ldots, 0]^{\top} \in \mathbb{R}^n$ be the vector whose $i$-th element is one and other elements are zero. The objective function in equation \eqref{equ4.8} can be restated as
\[
\sum_{i=1}^{n_1}\left\|\widehat{\mathcal{Y}}_{i}^{(r)}-\sum_{j=1}^{n_1} \widehat{\mathcal{W}}_{ij}^{(r)} \widehat{\mathcal{Y}}_{j}^{(r)}\right\|_F^2 = \sum_{i=1}^{n_1}\left\|\widehat{\mathcal{Y}}^{(r)\top} \widehat{e}_i - \widehat{\mathcal{Y}}^{(r)\top} \widehat{\mathcal{W}}_{i}^{(r)}\right\|_F^2,
\]
which can be formulated as
\begin{equation}
	\begin{gathered}	\sum_{i=1}^{n_1}\left\|\widehat{\mathcal{Y}}^{(r)\top} \widehat{e}_i-\widehat{\mathcal{Y}}^{(r)\top}\widehat{\mathcal{W}}_{i}^{(r)}\right\|_F^2=\left\|\widehat{\mathcal{Y}}^{(r)\top} \widehat{\mathcal{I}}_{n_1}^{(r)}-\widehat{\mathcal{Y}}^{(r)\top} \widehat{\mathcal{W}}^{(r)\top}\right\|_F^2 \\
		=\left\|\widehat{\mathcal{Y}}^{(r)\top}(\widehat{\mathcal{I}}_{n_1}^{(r)}-\widehat{\mathcal{W}}^{(r)})^\top\right\|_F^2,
	\end{gathered}
\end{equation}

where ${\mathcal{I}}_{n_1} \in \mathbb{R}^{n_1 \times n_1 \times n_3}$ is the Identity tensor, and $\widehat{\mathcal{W}}\in \mathbb{C}^{n_1 \times n_1 \times n_3}$.\\
By using tensor notation, the objective function in \eqref{equ4.8} can be formulated as
\[
\begin{aligned}
	\left\|\mathcal{Y}^{\top} \star (\mathcal{I}_{n_1}-\mathcal{W})^{\top}\right\|_F^2 & =n_3\operatorname{Trace}\left((\mathcal{I}_{n_1}-\mathcal{W}) \star \mathcal{Y} \star \mathcal{Y}^{\top} \star (\mathcal{I}_{n_1}-\mathcal{W})^{\top}\right) \\
	& =n_3\operatorname{Trace}\left(\mathcal{Y}^{\top} \star (\mathcal{I}_{n_1}-\mathcal{W})^{\top} \star (\mathcal{I}_{n_1}-\mathcal{W}) \star \mathcal{Y}\right) \\
	& =n_3\operatorname{Trace}\left(\mathcal{Y}^{\top} \star \mathcal{M} \star \mathcal{Y}\right),
\end{aligned}
\]
where 
\[ 
\mathcal{M}=(\mathcal{I}_{n_1}-\mathcal{W})^{\top} \star (\mathcal{I}_{n_1}-\mathcal{W}) \in \mathbb{R}^{n_1 \times n_1 \times n_3}.
\]
Note that the tensor $(\mathcal{I}_{n_1}-\mathcal{W})$ is the Laplacian of tensor $\mathcal{W}$. Then the tensor $\mathcal{M}$ can be considered as the gram tensor over the Laplacian of weight tensor. The second constraint will be satisfied implicitly in the optimization problem \eqref{equ4.8} see\cite{ghojogh2020locally}.
Then the optimization problem \eqref{equ4.8} can be formulated as
\begin{equation}\label{equ4.10}
	\underset{\mathcal{Y}}{\argmin } \ n_3{\operatorname { Trace }}\left(\mathcal{Y}^{\top} \star \mathcal{M} \star \mathcal{Y}\right),
	\text { subject to } \frac{1}{n_1} \mathcal{Y}^{\top} \star \mathcal{Y}=\mathcal{I}_{d}.
\end{equation}
{
We have the tensor $\mathcal{M}$ is f-symmetric then using the Theorem \ref{theo1}, the optimisation problem \eqref{equ4.10} is a generalized eigentube problem that is equivalent to}
\begin{equation}\label{equ4.11}
	\mathcal{M} \star \mathcal{Y}=\frac{1}{n_1 n_3}\mathcal{Y} \star \Lambda.
\end{equation}
The solution of the optimization problem \eqref{equ4.10} is the $d$ eigenslices associated to the $d$ smallest non-zero eigentube of the eigentube problem \eqref{equ4.11}.
Algorithm \ref{LME} shows a summary of the LME method.
\begin{algorithm}[H]
	\caption{Locally Multilinear Embedding}
	\label{LME}
	\begin{algorithmic}
		\STATE{\textbf{Input:} $\mathcal{X} \in \mathbb{R}^{n_1 \times n_2 \times n_3}, $ input data: third-order tensor,}
		\STATE{$d$: reduced dimension,}
		\STATE{$k$: number of neighborhoods,}

		\STATE{\textbf{Output:} $\mathcal{Y}\in \mathbb{R}^{n_1 \times d \times n_3}$,}\\
		\FOR{$r=1, \ldots,n_3$}
            \STATE{ From $\mathcal{X}^{(r)}$ build a k-NN graph,}
  		\FOR{$\mathcal{X}_{i}^{(r)} \in \mathcal{X}^{(r)\top}$}
            \STATE{Compute the $k \times k$ matrix $\boldsymbol{G}_{ir}$  
            \[
            \boldsymbol{G}_{ir}:=\left({\mathcal{X}}_{i}^{(r)} \mathbf{1}^{\top}-{\boldsymbol{X}}_{i}^{(r)}\right)^{\top}\left({\mathcal{X}}_{i}^{(r)} \mathbf{1}^{\top}-{\boldsymbol{X}}_{i}^{(r)}\right)\in \mathbb{R}^{k \times k},
            \] 
             ${\boldsymbol{X}}^{(r)}_{i}\in \mathbb{R}^{n_2 \times k}$ continent the $k$ neighbor of $\mathcal{X}_{i}^{(r)}$,
            }
            \STATE{ find the weights ${\mathcal{E}}^{(r)}_{i} = [{\mathcal{E}}^{(r)}_{ik}, \ldots, {\mathcal{E}}^{(r)}_{ik}]^{T} \in\mathbb{R}^k$ by solving $$
	\underset{\mathcal{E}_{i}^{(r)}}{\arg\min} \quad \sum_{i=1}^{n_1} \mathcal{E}^{(r) \top}_{i} \boldsymbol{G}_{ir} \mathcal{E}_{i}^{(r)},\ 
	\text{subject to} \quad \mathbf{1}^{\top} \mathcal{E}_{i}^{(r)} = 1.$$}
            \STATE{ Construct the $n_1 \times n_1$ frontal slice $\widehat{\mathcal{W}}^{(r)}$ given by
\[
\widehat{\mathcal{W}}_{ij}^{(r)}:= \begin{cases}\widehat{\mathcal{E}}^{(r)}_{i j} & \text { if } {\mathcal{X}}_{j}^{(r)}\in \text{k-NN}\left({\mathcal{X}}_{i}^{(r)}\right), \\ 0 & \text { otherwise.}\end{cases}
\] }
		\ENDFOR
            \ENDFOR
            \STATE{ ${\mathcal{W}}={\tt ifft}(\widehat{\mathcal{W}},[\,], 3)$, }
        \STATE{ $ 
\mathcal{M} = (\mathcal{I}_{n_1}-\mathcal{W})^{\top} \star (\mathcal{I}_{n_1}-\mathcal{W}) \in \mathbb{R}^{n_1 \times n_1 \times n_3}.
$}
        \STATE{$\mathcal{Y} \leftarrow $ Tensor diagonalization$\left(   \mathcal{M}, d\right)$,}
		\STATE{\textbf{return} $\mathcal{Y}$}.
	\end{algorithmic}
\end{algorithm}

\section{Numerical experiments}\label{s5}
In this part of our study, we evaluate the three techniques: Multilinear Local Discriminant Embedding (t-MLDE), Multilinear Laplacian Eigenmaps (t-MLE), and Locally Multilinear Embedding (t-LME). We compare them with state-of-the-art methods \cite{dufrenois2022multilinear, lu2008mpca,yan2006multilinear}. Each method is applied ten times to four different datasets to evaluate their effectiveness in reducing dimensionality. After that, we calculate the mean time required for the dimensionality reduction process. In our approach, we use random forest algorithm for classification, after that we use cross-validation to determine the average accuracy, as suggested by Berrar see \cite{berrar2019cross}. For this, we use 80\% of each dataset for training and the remaining 20\% for testing. Regarding the selection of $k_1$, $k_2$, $t_1$, $t_2$ in MLDE, $k$, $t$ in MLaplacian Eigenmaps, and $k$ in Locally Multilinear Embedding we experimented with various values using a small subset of data. After determining the most effective values, we then applied these to the entire dataset.\\
\subsection{Data sets}
In our experiments, we used four multidimensional databases: face recognition AR, FEI, Brain Tumor MRI, and COVID-19 Chestxray. The AR and FEI datasets are specifically designed for face recognition, while the Brain Tumor MRI and COVID-19 datasets provide MRI and X-ray images, respectively, which have applications in medical imaging. In the following, we have a description of datasets.\\
\begin{itemize}
    \item \textbf{AR database}: The AR database contains 2600 images featuring frontal faces with different expressions, lighting conditions, and occlusions. where each subject has 26 facial images taken in two sessions separated by two weeks, as illustrated in Fig. Furthermore, two formats of data representations are used according to the formulations of algorithms. By 2D formulation, an image is stored as a matrix of size 115 by 115. we resize each image to 32 by 32 pixels and the dimension of the dataset is 2600×32×32×3.\\
    \item \textbf{FEI face databese}: The FEI database is a Brazilian face database that contains a set of face images taken between June 2005 and March 2006 at the Artificial Intelligence Laboratory of FEI in São Bernardo do Campo, São Paulo, Brazil. There are 14 images for each of 200 individuals, corresponding to 2800 images in total. All images are colorful and taken against a white homogenous background in an upright frontal position with profile rotation of up to about 180 degrees. Scale might vary about 10\% and the original size of each image is 640x480 pixels. we resize each image to 32 by 32 pixels and the dimension of the dataset is 2800×32×32×3.\\
    \item \textbf{Brain Tumor MRI Dataset} The Brain Tumor MRI Dataset contains 1311 images of human brain MRI images, which are classified into 4 classes glioma, meningioma, no tumor, and pituitary. The dimensions of this data are (1311, 32, 32, 3), where 1311 is the number of MRI images, 64 is the height and width of each image in pixels, and 3 represents the RGB color channels. This dataset provides a rich source of medical imaging data for research into the automatic detection and classification of brain tumors and can be used to train and evaluate machine learning algorithms in this domain. The use of MRI in the detection of brain tumors is a well-established medical imaging technique, and this dataset allows for further development and improvement of these techniques. The results of research using this dataset have the potential to improve patient outcomes and save lives.\\
    \item \textbf{COVID-19}: The COVID-19 is a collection of chest X-ray images of patients with or without COVID-19 and/or pneumonia. The dataset is classified into three classes  COVID-19, Pneumonia, and Normal, and has dimensions of (925, 32, 32, 3), with 925 images of 32x32 pixels and 3 color channels. The goal of the dataset is to provide a large and diverse set of data for research and development of machine learning algorithms for the automatic detection of COVID-19 and pneumonia on chest X-rays.
\end{itemize}
\begin{center}
\begin{figure}[!tbp]
	\centering
	\begin{minipage}[b]{0.49\textwidth}
		\includegraphics[width=0.5\textwidth]{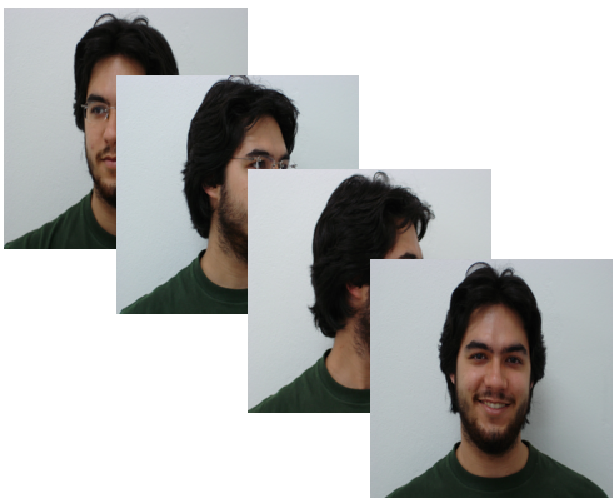}
		\caption{ FEI data set.}
	\end{minipage}
	\hfill
	\begin{minipage}[b]{0.49\textwidth}
		\includegraphics[width=0.5\textwidth]{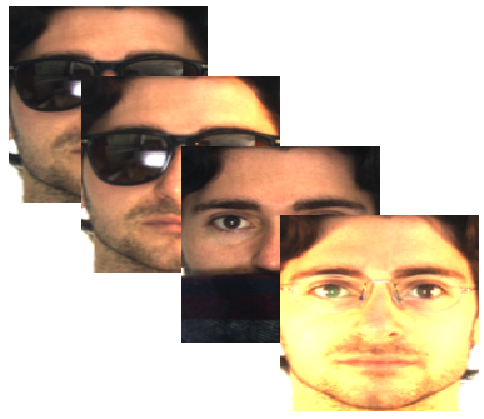}
		\caption{ AR data set.}
	\end{minipage}
\end{figure}
\begin{figure}[!tbp]
	\centering
	\begin{minipage}[b]{0.49\textwidth}
		\includegraphics[width=0.5\textwidth]{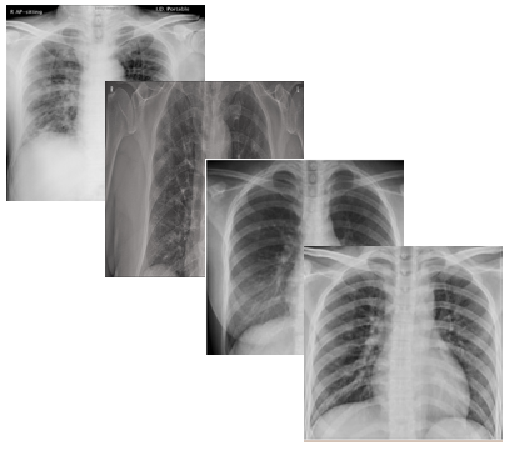}
		\caption{ Novel COVID-19 Chestxray Repository.}
	\end{minipage}
	\hfill
	\begin{minipage}[b]{0.49\textwidth}
		\includegraphics[width=0.5\textwidth]{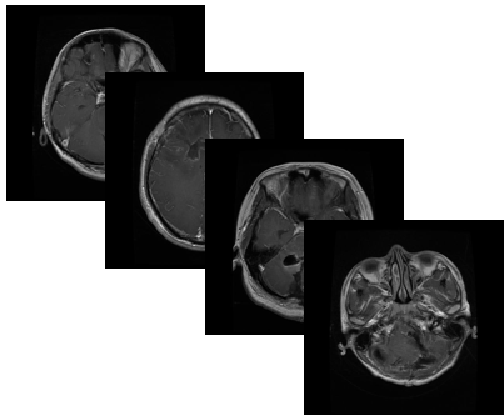}
		\caption{ Brain Tumor MRI Dataset.}
	\end{minipage}
\end{figure}
\end{center}
\subsection{Results and discussion}
        
        
        
        
        

\begin{figure}[!tbp]
	\centering
	\includegraphics[width=1\textwidth]{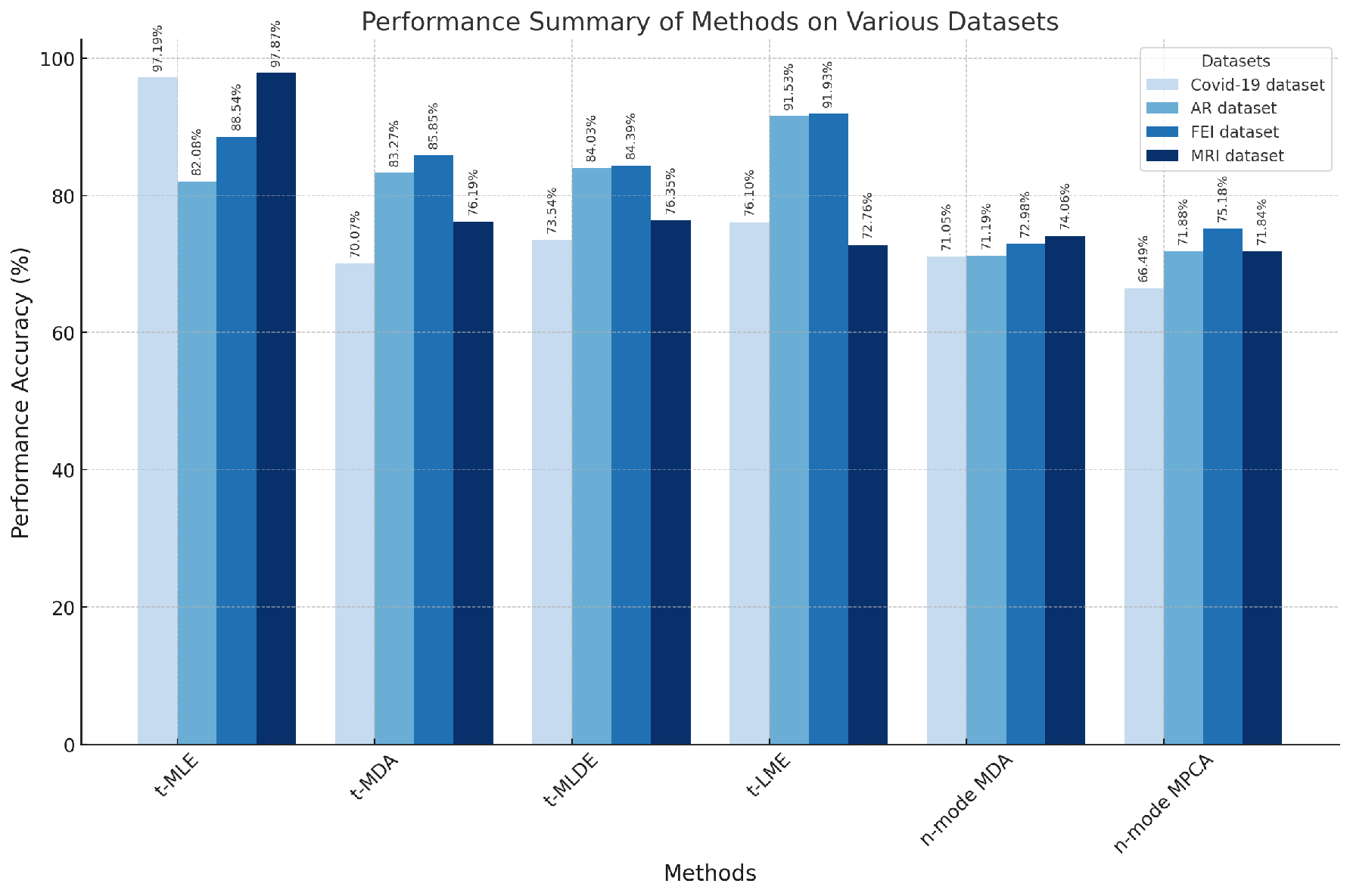}
	\caption{Summary of performance accuracy \% of methods across various datasets.}
	\label{Figure5.5}
\end{figure}
\begin{table}[H]
    \centering
    \begin{tabular}{l|llllll}
        \hline
        Methods & Covid-19 dataset & AR dataset & FEI dataset & MRI dataset \\
        \hline
        t-MLE & $1.86$ & $41.04$ & $35.5$ & $5.51$ \\
        
        t-MDA & $65.20$ & $58.57$ & $65.63$ & $58.07$ \\
        
        t-MLDE & $42.33$ & $154.92$ & $185.07$ & $54.42$ \\
        
        t-LME & $69.59$ & $330.91$ & $389.06$ & $100.10$ \\
        
        n-mode MDA & $4.53$ & $12.52$ & $16.02$ & $5.50$ \\
        
        n-mode MPCA & $13.82$ & $42.54$ & $36.92$ & $20.15$ \\
        \hline
    \end{tabular}
    \caption{Time complexity in seconds for dimensionality reduction across various datasets.}
    \label{table:2}
\end{table}

Figure \ref{Figure5.5} provide a summary of method performances on the AR, FEI, COVID-19, and Brain Tumor MRI datasets. This figure showcases the average accuracy values attained by t-MLE, t-MLDE, t-LME and t-MDA proposed in \cite{dufrenois2022multilinear}, n-mode MPCA proposed in \cite{lu2008mpca}, and n-mode MDA proposed in \cite{yan2006multilinear}.

Firstly, we can see from Figure \ref{Figure5.5} that methods using the t-product surpass those based on the n-mode product in terms of accuracy. This is clearly demonstrated by the comparison between {t-MDA} and n-mode MDA. This difference can be attributed to the computational approach required for each method. Specifically, the n-mode product-based methods necessitate first converting the tensor into a matrix (metricizing) before calculating the eigenvectors. Conversely, t-product-based methods allow for the direct computation of the solution across the entire tensor without the need to metricize the tensor first.

Secondly, also from the figure \ref{Figure5.5}, we notice that when comparing methods based on the t-product, t-MLE and t-LME perform better in terms of accuracy than t-MDA and t-MLDE. This difference in performance can be attributed to the fact that t-MDA and t-MLDE are linear, whereas t-MLE and t-LME are non-linear. Generally, non-linear methods tend to be more accurate.  

Thirdly, when comparing t-MLE with t-LME, it's observed that t-LME performs better with datasets containing a large number of classes, offering higher accuracy. For example, the AR dataset, which contains 26 classes, shows t-LME achieving an accuracy of 91.53\% compared to t-MLE's 82.08\%. Conversely, for datasets with a smaller number of classes, such as the COVID-19 dataset, which contains just 3 classes, t-MLE shows superior results. Specifically, the accuracy of t-MLE on the COVID-19 dataset is 97.19\%, whereas t-LME achieves only 76.10\% accuracy on the same dataset. Therefore, understanding the number of classes in our dataset allows us to select the method that provides the best accuracy.

Table \ref{table:2} provides an overview of the computational complexities associated with dimensionality reduction across various datasets, including FEI, AR, COVID-19, and brain tumor MRI, using different methods.

Firstly note that the t-LME method requires a significant time for dimensionality reduction, which is expected given its complexity. The process starts with calculating the weights for the multi-linear reconstruction of each data point in the high-dimensional input space. This calculation necessitates the resolution of an optimization problem. Following this, the method projects the data points into a lower-dimensional embedding space, using the previously determined weights. This projection phase also requires the solution of a tarce-ratio tensor problem.

Secondly, it is observed that the t-MLE method demonstrates remarkable speed in processing. For instance, with the COVID-19 dataset, it completes the task in merely 1.86 seconds, and for the MRI dataset, it requires only 5.51 seconds.

In summary, while both t-MLE and t-LME methods are effective for reducing the dimensionality of multidimensional data, However, t-MLE stands out as the better choice, not only for its rapid processing speed but also for its accuracy.

\section{Conclusions}\label{s6}
In this paper, we propose a generalization of trace ratio methods for multidimensional data. This generalization includes Local Discriminant Embedding, Laplacian Eigenmaps, and Locally Linear Embedding based on the concept of t-product. 
To extend these methods to tensors or multilinear data, we present certain definitions and propose theoretical findings. Additionally, we offer a Newton-QR algorithm as a solution to the trace-ratio challenge. Finally, we showcase the numerical results of these methods compared to the state-of-the-art.

\end{document}